\documentclass[11pt,reqno]{amsart}

\usepackage{graphics,graphicx,fontenc,mathabx}
\usepackage{epsfig,psfrag,manfnt}
\usepackage{bbm,subfigure,rotating,bm,float,wasysym}
\usepackage[all,cmtip]{xy}
\usepackage{amssymb,amsmath,amsfonts,amsthm,color}

\usepackage{scalerel,stackengine}
\stackMath
\newcommand\reallywidehat[1]{%
\savestack{\tmpbox}{\stretchto{%
  \scaleto{%
    \scalerel*[\widthof{\ensuremath{#1}}]{\kern-.6pt\bigwedge\kern-.6pt}%
    {\rule[-\textheight/2]{1ex}{\textheight}}
  }{\textheight}%
}{0.5ex}}%
\stackon[1pt]{#1}{\tmpbox}%
}

\makeatletter
\newcommand*\bigcdot{\mathpalette\bigcdot@{.5}}
\newcommand*\bigcdot@[2]{\mathbin{\vcenter{\hbox{\scalebox{#2}{$\m@th#1\bullet$}}}}}
\makeatother

\newcommand{\calD}{\mathcal{D}}
\newcommand{\calE}{\mathcal{E}}

\newcommand{\calS}{\mathcal{S}}

\newcommand{\mC}{\mathbb{C}}

\newcommand{\mN}{\mathbb{N}}

\newcommand{\mR}{\mathbb{R}}

\newcommand{\mT}{\mathbb{T}}

\newcommand{\mZ}{\mathbb{Z}}

\newcommand{\bbH}{\mathbf{H}}

\newcommand{\bbT}{\mathbf{T}}

\newcommand{\bsigma}{\bm{\sigma}}
\newcommand{\bSigma}{\bm{\Sigma}}

\newtheorem{theorem}{Theorem}[section]
\newtheorem{lemma}[theorem]{Lemma}

\newtheorem{proposition}[theorem]{Proposition}

\theoremstyle{definition}

\theoremstyle{definition}
\newtheorem{definition}[theorem]{Definition}

\theoremstyle{definition}
\newtheorem{notation}[theorem]{Notation}

\theoremstyle{definition}

\begin{document}



\title[Summation method and an example]{A summation method based on \\
the Fourier series of periodic distributions \\
and an example arising in the Casimir effect}

\author[A. Sasane]{Amol Sasane}
\address{Department of Mathematics \\London School of Economics\\
    Houghton Street\\ London WC2A 2AE\\ United Kingdom}
\email{A.J.Sasane@lse.ac.uk}

\begin{abstract} 
A generalised summation method is considered based on the Fourier 
series of periodic distributions.  It is shown that
$$
e^{it}-2e^{2it}+3e^{3it}-4e^{4it}+-\cdots
=
{\mathrm P\mathrm f} {\displaystyle \frac{e^{it}}{(1+e^{it})^2}} 
+i\pi \displaystyle \sum_{n\in \mZ} \delta'_{(2n+1)\pi},
$$
where 
${\mathrm P\mathrm f} {\displaystyle \frac{e^{it}}{(1+e^{it})^2}}\in \calD'(\mR)$ 
is the $2\pi$-periodic distribution given by 
\begin{eqnarray*}
\left\langle \!{\mathrm P\mathrm f} {\displaystyle \frac{e^{it}}{(1\!+\!e^{it})^2}} ,\varphi \!\right\rangle
\!\!\!\!\!&=&\!\!\!\!\!
 \lim_{\epsilon\searrow 0} 
 \left( \int_{(-\delta,\pi-\epsilon)\cup(\pi+\epsilon,2\pi+\delta)}\! \frac{\varphi(t) e^{it}}{(1\!+\!e^{it})^2}dt \!-\!\frac{\varphi(\pi)}{\tan (\epsilon/2)}\right)
\end{eqnarray*}
for $ \varphi \in \calD(\mR)$ with support $\textrm{supp}(\varphi)\subset (-\delta,2\pi+\delta)$, where $\delta\in (0,\pi)$. 
Applying the generalised summation
method, the sum of the divergent series $1+2+3+\cdots$ 
(arising in a quantum field theory calculation in the Casimir effect) is determined,
and more generally also the sum $1^k+2^k+3^k+\cdots$, for $k\in \mN$, is determined. 
\end{abstract}

\maketitle

\section{Introduction}

\noindent A quantum field theoretic calculation predicts the so-called
Casimir effect, giving the correct\footnote{by now experimentally verified
\cite{Ede}} value of the attractive force between two parallel perfect
conductor plates in vacuum, if one uses the absurd sum
\begin{equation}
\label{eq_16_01_2020_12_32}
\sum_{n=1}^\infty n =-\frac{1}{12}.
\end{equation}
While there exists a generalised summation method (zeta function
regularisation) that allows one to show this, it is somewhat contrived
in the context of the Casimir effect calculation. On the other hand,
Fourier series makes a direct appearance in the quantum field
calculation (see e.g. \cite[\S III]{KNW}), since one uses the quantised momentum space.  We give an
alternative summation method, based on distributional Fourier series,
which gives a derivation of \eqref{eq_16_01_2020_12_32}, and gives a
rigorous mathematical justification of \eqref{eq_16_01_2020_12_32}.

\noindent 
A justification for 
$$
1+2+3+\cdots=-\frac{1}{12},
$$
for example\footnote{There are older justifications of this. For example, in \cite{Son}, 
the determination of $\zeta(-1)$ is attributed to Euler. The related summation of the divergent 
  series $1-2+3-+\cdots$, being assigned the sum $1/4$, can be based on the so-called {\em Euler summation method}. (For a
  modern discussion of the Euler summation method, see for example \cite[\S20,
  p.325-328]{Kor}.)  Euler
  had discovered a functional equation relating the zeta
  function with the Dirichlet eta function (alternating zeta function)
  for integral values, which yields, using the alternating sum
  $1-2+3-+\cdots=1/4$, also that $\zeta(-1)=1+2+3+\cdots=-1/12$
  \cite[Vol. 14, p. 442-443, 594-595; Vol. 15, p. 70-90]{Eul}.   The determination of the zeta function at all negative
  integers, using the Euler summation method, can be found in
  \cite{Son}.} given by Ramanujan in one of his notebooks
\cite[Ch.6, p.135]{Ber}, is as follows:

\bigskip 

\noindent 
If $s:=1+2+3+\cdots$, then formally 
\begin{equation}
\label{ramanujan_manipulation}
\begin{array}{rcl}
s&=&1+2+3+4+5+6+\cdots\\
-4s&=&\phantom{1}-4\phantom{+3+}-8\phantom{+5}-12\cdots \\ \hline 
-3s&=& 1-2+3-4+5-6+-\cdots.
\end{array}
\end{equation}
The power series expansion 
$$
\frac{1}{(1+x)^2}= 1-2x+3x^2-4x^3+-\cdots \textrm{ for }|x|<1
$$
motivates associating the sum of the alternating series
$1-2+3-4+-\cdots$ with
$$
\displaystyle \left. \frac{1}{(1+x)^{2}}\right|_{x=1},
$$ 
and one writes 
\begin{equation}
\label{eq_16_01_2020_18_33}
1-2+3-4+-\cdots=\frac{1}{4}.
\end{equation}
In light of the last equation from \eqref{ramanujan_manipulation},
together with \eqref{eq_16_01_2020_18_33}, we arrive at
\begin{equation}
\label{eq_16_01_2020_18_35}
s=1+2+3+\cdots=-\frac{1}{3}\cdot \frac{1}{4}=-\frac{1}{12}.
\end{equation}
There exist generalised summation techniques which allow one to obtain
the result \eqref{eq_16_01_2020_18_35}, for example the zeta function
regularisation method \cite[p.301]{Zei}. We briefly recall the idea behind the zeta function 
regularisation method in
Appendix B, in order to contrast it with our summation method introduced here. 

In this article, we give an alternative route to
obtaining \eqref{eq_16_01_2020_18_35}, based on the Fourier series of
periodic distributions.  Distributional Fourier series converge in the
sense of distributions whenever the Fourier coefficients $c_n$ grow at
most polynomially, that is, they satisfy for some $M>0$ and $k>0$ that
$$
\textrm{ for all }n\in \mZ,\,\,|c_n|\leq M(1+|n|)^k,
$$
and so one could try making the above manipulations
\eqref{ramanujan_manipulation} with divergent series on a firmer
footing by using the Fourier theory of distributions, where the place
holders of the terms of the divergent series are naturally provided by
the characters $e^{int}$, and the operation in the second step of
\eqref{ramanujan_manipulation} can be carried out by simply doubling
the period, that is, by a homothetic dilation by $2$ of the
distributional Fourier series. The question then arises if one can
give an explicit formula for the distribution corresponding to the
Fourier series (corresponding to the right hand side of the last line in
\eqref{ramanujan_manipulation})
$$
e^{it}-2e^{2it}+3e^{3it}-4e^{4it}+-\cdots.
$$
The aim of this article is mainly to carry out this computation, and
we give an explicit expression for this distribution.  We also give an
operation which can be justifiably thought of as being the operation
of `setting $t=0$' in such a distributional Fourier series. Within the
framework of our generalised summation method, this allows us to
arrive at
$$
\sum_{n=1}^\infty n =-\frac{1}{12}.
$$
We remark that this last series plays a role in the Casimir effect in
quantum field theory \cite[\S6.6]{Zei}. We briefly elaborate on this
link in Appendix A.

The paper is organised as follows. We introduce the summation method
in Section \ref{summation_method}.  In the Sections \ref{section_3}
and \ref{section_4}, we construct a certain $2\pi$-periodic
distribution $S$ on $\mR$, and in Section \ref{sect_FC}, we determine its
Fourier coefficients. In Section \ref{section_5}, it is shown that $S$
has a Fourier series that is summable at $t=0$. Finally, in Sections
\ref{section_6} and \ref{section_7}, we apply our generalised
summation method to determine $1+2+3+\cdots$, and more generally,
$1^k+2^k+3^k+\cdots$ for any $k\in \mN$.

\section{Summation method}
\label{summation_method}

\noindent 
For the background on the Fourier series theory of periodic
distributions, we refer the reader to \cite[Chap.IV]{Sch},
\cite[p.91,101-104]{Zui} and \cite[Chap.33]{Don}. Throughout this
article, unless otherwise indicated, we will use the standard
distribution theory notation from Schwartz \cite{Sch1} or Tr\'eves
\cite{Tre}.  We recall the basic definitions below.

\begin{definition}[Translation operator $\bbT_a$; periodic distributions]$\;$

\noindent 
 Let $a\in \mR$. 
 \begin{itemize}
  \item 
For $f:\mR\rightarrow \mR$, $\bbT_a f:\mR\rightarrow \mR$  is defined by 
 $$
 (\bbT_a f)(t)=f( t-a),\quad t\in \mR.
 $$
 \item 
 For $T\in \calD'(\mR)$,  $\bbT_a T\in \calD'(\mR)$ is defined by 
 $$
 \langle \bbT_a T,\varphi \rangle =  \langle T, \bbT_{-a}\varphi\rangle, 
 \quad \varphi \in \calD(\mR).
 $$
 \item 
 Let $\tau>0$. A distribution $T\in \calD'(\mR)$ is said to be {\em $\tau$-periodic }
 if $\bbT_\tau T=T$. 
 \end{itemize}
\end{definition}

\noindent For computational ease, we will now work with a
$2\pi$-periodic distribution. A $2\pi$-periodic distribution
$T\in \calD'(\mR)$ possesses a Fourier series expansion
$$
T=\lim_{N\rightarrow \infty}\sum_{n=-N}^N c_n(T) e^{int}=:\sum_{n\in \mZ} c_n(T) e^{int},
$$
where the $c_n(T)$ are complex numbers, and the convergence is in
$\calD'(\mR)$. The Fourier coefficients $c_n(T)$ are best expressed in
terms of the distribution
$T_{\scriptscriptstyle \textrm{circle}}\in \calD'(\mT)$, where
$\mT:=\{z\in \mC:|z|=1\}$, obtained by `wrapping $T$ on the
circle'. We make this precise below.

For a $\varphi \in \calD(\mR)$, we first define the $2\pi$-periodic
smooth function $\varphi_{\scriptscriptstyle\textrm{circle}}$
$$
\varphi_{\scriptscriptstyle\textrm{circle}}(e^{it})=\sum_{n\in \mZ} \varphi(t-2\pi n ), \quad t\in \mR.
$$
We note that at each point $t\in \mR$, only finitely many terms in the
series above are nonzero.  Then given a
$T_{\scriptscriptstyle\textrm{circle}}\in \calD'(\mT)$, this defines a
$2\pi$-periodic distribution $T\in \calD'(\mR)$ as follows:
$$
\langle T,\varphi\rangle =\langle T_{\scriptscriptstyle \textrm{circle}}, \varphi_{\scriptscriptstyle \textrm{circle}}\rangle, \quad \varphi \in \calD(\mR).
$$
It can be shown that this map is a bijection between $\calD'(\mT)$ and
the subspace of $\calD'(\mR)$ consisting of all $2\pi$-periodic
distributions; see \cite[p.150-151]{Sch}.  The Fourier coefficients
$c_n(T)$ are given by
$$
c_n(T)=\frac{1}{2\pi} \langle T_{\scriptscriptstyle\textrm{circle}}, e^{-int}\rangle,\quad n\in \mZ.
$$

\noindent We now consider a sequence of pointwise nonnegative test functions
$(\varphi_m)_{m\in \mN}$, with shrinking supports, and unit area,
which provides an approximation to the Dirac delta distribution
$\delta_0$ with support at $0$.

\begin{definition}[Symmetric positive mollifier, approximate identity]$\;$ 

\noindent 
A test function $\varphi\in \calD(\mR)$ is said to be a {\em symmetric positive  mollifier} if 
\begin{itemize}
\item $\varphi(t)=\varphi(-t) $ for all $t\in \mR$, $\phantom{\displaystyle\int_\mR \varphi(t)dt=1 }$
\item $\varphi(t)\geq 0$ for all
  $t\in \mR$,$\phantom{\displaystyle\int_\mR \varphi(t)dt=1 }$
\item $\displaystyle \int_\mR \varphi(t)dt=1$.
\end{itemize}
Define the sequence $(\varphi_m)_{m\in \mN}$ by
$$
\varphi_m(t)=m\varphi(mt),\quad t\in \mR, \;\; m\in \mN.
$$
Then $\varphi_m\rightarrow \delta_0$ as $m\rightarrow \infty$ in
$\calD'(\mR)$, where for $a\in \mR$ the notation $\delta_a$ is used
for the Dirac distribution on $\mR$ with support $\{a\}$. 

\smallskip 

\noindent We
call $(\varphi_m)_{m\in \mN}$ an {\em approximate identity}.
\end{definition}

\goodbreak

\noindent For a $2\pi$-periodic distribution $T$, we have 
$$
T=\sum_{n\in \mZ} c_n(T) e^{int},
$$
and since this is an equality of distributions, we cannot in general
`set $t=0$'. However, we could look at an approximate identity
$(\varphi_m)_{m\in \mN}$, and the limit
$$
\lim_{m\rightarrow \infty} \langle T,\varphi_m\rangle,
$$
if it exists, can be thought of as `setting $t=0$ in $T$'. 
For example, we have the following elementary result.

\begin{proposition}
Let $f\in C(\mR\setminus  \{0\})$ be such that the limits 
$$
f(0+):=\lim_{x\searrow 0} f(x) \textrm{ and } f(0-):=\lim_{x\nearrow 0} f(x) 
$$
exist, and let $T_f\in \calD'(\mR)$ be the 
distribution given by 
$$
\langle T_f,\psi\rangle:=  \int_{\mR} f(t) \psi(t)dt \textrm{ for } 
\psi \in \calD(\mR).
$$
Then for any approximate identity $(\varphi_m)_{m\in \mN}$, we have 
$$
\displaystyle 
\lim_{m\rightarrow \infty} \langle T_f,\varphi_m\rangle=\frac{f(0+)+f(0-)}{2}.
$$
\end{proposition}
\begin{proof} We have 
\begin{eqnarray*}
 \langle T_f,\varphi_m\rangle &=& \int_{\mR}f(t) \varphi_m(t)dt\\
 &=&\int_{-\infty}^\infty f(t) m\varphi(mt)dt \\
 &=&\int_{-\infty}^\infty f\left(\frac{\tau}{m}\right)\varphi(\tau )d\tau\\
 &=& \int_{-\infty}^0 f\left(\frac{\tau}{m}\right)\varphi(\tau )d\tau+\int_0^{\infty} f\left(\frac{\tau}{m}\right)\varphi(\tau )d\tau.
\end{eqnarray*}
By the Lebesgue dominated convergence theorem, it follows that 
\begin{eqnarray*}
\lim_{m\rightarrow \infty} \langle T_f,\varphi_m\rangle
&=&\int_{-\infty}^0 f(0+)\varphi(\tau )d\tau+ \int_0^{\infty} f(0-)\varphi(\tau )d\tau
\\
&=&
f(0+)\int_{-\infty}^0 \varphi(\tau )d\tau+f(0-) \int_0^{\infty} \varphi(\tau )d\tau\\
&=& f(0+)\cdot \frac{1}{2}\int_{\mR}\varphi(\tau )d\tau+f(0-) \cdot \frac{1}{2}\int_{\mR} \varphi(\tau )d\tau
\\
&=&
\frac{f(0+)+f(0-)}{2},
\end{eqnarray*}
where we used the fact that $\varphi$ is a symmetric positive mollifier. 
\end{proof}

\begin{definition} 
\label{def_summable}$\;$

\noindent 
For a $2\pi$-periodic distribution $T$, let
$$
\displaystyle 
\sum_{n\in \mZ} c_n(T)e^{int}=T
$$ 
in $\calD'(\mR)$.  

\medskip 

\noindent If there exists a $\sigma\in \mC$ such that for any
approximate identity $(\varphi_m)_{m\in \mN}$, the limit
$\displaystyle \lim_{m\rightarrow \infty} \langle T, \varphi_m\rangle$
exists, and
 $\displaystyle 
\lim_{m\rightarrow \infty} \langle T, \varphi_m\rangle=\sigma, 
$ 
then we say {\em the series}
$$
\displaystyle \sum_{n\in \mZ} c_n(T)
$$ 
{\em is summable}, or the {\em Fourier series of $T$ is summable at
  $t=0$}, and we define
$$
\sum_{n\in \mZ} c_n(T)=\sigma.
$$
We call a double-sided complex sequence $(C_n)_{n\in \mZ}$ {\em
  summable} if there exists a $2\pi$-periodic distribution $T\in \calD'(\mR)$ such
that
\begin{itemize}
\item for all $n\in \mZ$, $C_n=c_n(T)$, and $\phantom{\displaystyle \int}$
\item the Fourier series of $T$ is summable at $t=0$.$\phantom{\displaystyle \int}$
\end{itemize}
(Since for any $2\pi$-periodic distribution the Fourier coefficients
are unique, this last definition is a well-defined notion.)
\end{definition}

\begin{notation}[$\bsigma$] $\;$

\noindent 
  We denote the set of all $2\pi$-periodic distributions that have a Fourier series that is
   summable at $t=0$, in the sense of Definition~\ref{def_summable}, by
  $\bsigma$.
\end{notation}

\noindent We have the following trivial observation that $\bsigma$ is a
subspace of the space of all $2\pi$-periodic distributions in $\calD'(\mR)$.

\begin{proposition}
\label{prop_linearity}
If $T,S\in \calD'(\mR)$ are $2\pi$-periodic distributions belonging to
$\bsigma$, then for any $\alpha,\beta\in \mC$, also $\alpha T+\beta S$
also belongs to $\bsigma$, and moreover
$$
\sum_{n\in \mZ} c_n(\alpha T +\beta S)
=
\alpha \sum_{n\in \mZ} c_n(T)+ \beta \sum_{n\in \mZ} c_n(S) .
$$
\end{proposition}
\begin{proof} This follows immediately from
  Definition~\ref{def_summable} of summability and the linearity of
  the $n$th Fourier coefficient on the space of all $2\pi$-periodic
  distributions.
\end{proof}

\noindent Let us now show that $\bsigma$ is strictly contained in the
subspace of all $2\pi$-periodic distributions on $\mR$.

\begin{proposition}
\label{prop_counter_example_17_jan_2020_1139}
Let $T$ be the $2\pi$-periodic distribution
$ T=\displaystyle \sum_{n\in \mZ} e^{int}.  $
 
\noindent 
Then $T\not\in \bsigma$.
\end{proposition}
\begin{proof}
  Let $\varphi\in \calD(\mR)$ be a symmetric positive mollifier. 
Moreover, suppose that $\textrm{supp}(\varphi)\subset (-2\pi,2\pi)$
and $\varphi(0)>0$.  For $m\in \mN$, with
$\varphi_m:=m\varphi(m\cdot)$, we have
\begin{eqnarray}
\nonumber 
\langle T, \varphi_m\rangle 
&=&\sum_{n=-\infty}^\infty \langle e^{int},\varphi_m\rangle \\
\nonumber 
&=&\sum_{n=-\infty}^\infty \int_{-\infty}^\infty e^{int} \varphi_m(t) dt\\ 
\label{eq_17_jan_2020_11_13}
&=&\sum_{n=-\infty}^\infty \reallywidehat{\varphi_m} (-n),
\end{eqnarray}
where $\reallywidehat{\varphi_m}$ denotes the Fourier transform of
$\varphi_m$. By the Poisson summation formula for
$\varphi_m \in \calD(\mR)\subset \calS(\mR)$ (where $\calS(\mR)$
denotes the Schwartz space of test functions), we have
\begin{equation}
\label{eq_17_jan_2020_11_17}
\frac{1}{2\pi}\sum_{k=-\infty}^\infty \reallywidehat{\varphi_m}(k)
=
\sum_{n=-\infty}^\infty \varphi_m(2\pi n).
\end{equation}
Using \eqref{eq_17_jan_2020_11_17} in \eqref{eq_17_jan_2020_11_13}, we
obtain
\begin{eqnarray*}
\langle T, \varphi_m\rangle 
&=&\sum_{n=-\infty}^\infty \reallywidehat{\varphi_m} (-n)\\
&=&
\sum_{k=-\infty}^\infty \reallywidehat{\varphi_m} (k) \quad\quad (\textrm{substituting }k:=-n)
\\
&=& 2\pi \sum_{n=-\infty}^\infty \varphi_m(2\pi n)\\
&=&2\pi m\sum_{n=-\infty}^\infty \varphi(2\pi mn)\\
&=&2\pi m \varphi(0),\phantom{\sum_{n=-\infty}^\infty }
\end{eqnarray*}
where we have used the fact that $m\in \mN$ and that
$\textrm{supp}(\varphi)\subset (-2\pi,2\pi)$ in order to get the last
equality.  But then as $\varphi(0)>0$, we have
$$
\lim_{m\rightarrow \infty} \langle T,\varphi_m\rangle =2\pi \varphi(0) \cdot 
\lim_{m\rightarrow \infty} m= +\infty,
$$
and so $T\not \in \bsigma$. 
\end{proof}

\noindent In order to `insert zeroes' between the series terms (in our
case between the Fourier series terms), as in the operation \`a la
Ramanujan in \eqref{ramanujan_manipulation} (when the $-4s$ sum is
matched with the $s$ sum in a particular manner), we will consider
$\bbH_2 T$ for a periodic distribution $T$, where $\bbH_2 T$ is
obtained from $T$ by dilation by a factor of $2$, hence `doubling the
period'.

\goodbreak 

\begin{definition}[Homothetic transformation]
Let $\lambda>0$. 
\begin{itemize}
\item For $f:\mR\rightarrow \mR$, $\bbH_\lambda f:\mR\rightarrow \mR$
  is defined by
 $$
 (\bbH_\lambda f)(t)=f(\lambda t),\quad t\in \mR.
 $$
\item For $T\in \calD'(\mR)$, $\bbH_\lambda T\in \calD'(\mR)$ is
  defined by
 $$
 \langle \bbH_\lambda T,\varphi \rangle = \frac{1}{\lambda} \langle T, \bbH_{1/\lambda}\varphi\rangle, 
 \quad \varphi \in \calD(\mR).
 $$
\end{itemize}
\end{definition}

\begin{proposition}
\label{prop_hom}
Let $T$ be a $2\pi$-periodic distribution
$ T=\displaystyle \sum_{n\in \mZ} c_n(T)e^{int}.  $ Then the
$2\pi$-periodic distribution $\bbH_2 T$ has the `lacunary' Fourier
series
$$
\bbH_2 T=\sum_{n\in \mZ} c_n(T) e^{2int}.
$$
Moreover, $\bbH_2 T\in \bsigma$ if and only if $T\in
\bsigma$. Furthermore if $T\in \bsigma$, then we also have that
$$
\sum_{n\in \mZ} c_n(T)=\sum_{n\in \mZ} c_n(\bbH_2 T).
$$
\end{proposition}
\begin{proof} The map
  $\bbH_\lambda :\calD'(\mR)\rightarrow \calD'(\mR)$ is continuous,
  and so it follows by a termwise application of $\bbH_2$ on the
  Fourier series expansion of $T$ that
$$
\bbH_2 T=\sum_{n\in \mZ} c_n(T) e^{2int}.
$$
Now suppose that $(\varphi_m)_{m\in \mN}$ is an approximate
identity. Then we have that $\varphi_m=m\varphi(m\cdot)$, for a symmetric positive mollifier 
$\varphi\in\calD(\mR)$. 
But then $\psi\in \calD(\mR)$, given by
$$
\psi=\displaystyle \frac{1}{2} \bbH_{\frac{1}{2}} \varphi ,
$$
is also a symmetric positive mollifier. 
Hence $(\psi_m)_{m\in \mN}$, where $\psi_m=m\psi(m\cdot)$, is also an
approximate identity.
  
  \medskip

\noindent 
If $ T\in \bsigma$, then
\begin{equation}
\sum_{n\in \mZ}c_n(T)
= \lim_{m\rightarrow \infty}\langle T,\psi_m\rangle 
=\lim_{m\rightarrow \infty}\frac{1}{2} \langle T,\bbH_{\frac{1}{2}} \varphi_m \rangle 
\label{eq_16_jan_2020_19_05}
=\lim_{m\rightarrow \infty} \langle \bbH_2 T,\varphi_m \rangle , 
\end{equation}
and so $\bbH_2 T\in \bsigma$ too, and moreover,
\begin{equation}
\label{eq_16_jan_2020_19_07}
\sum_{n\in \mZ} c_n(T)=\sum_{n\in \mZ} c_n(\bbH_2 T).
\end{equation}
 
\smallskip
 
\noindent
Now suppose that $\bbH_2 T\in \bsigma$. Suppose that $(\psi_m)_{m\in \mN}$ is an approximate
identity. Then $\psi_m=m\psi(m\cdot)$, for a symmetric positive mollifier $\psi\in\calD(\mR)$. 
But then $\varphi\in \calD(\mR)$, given by
$$
\varphi=\displaystyle 2 \bbH_{2} \psi ,
$$
is also a symmetric positive mollifier. 
Hence $(\varphi_m)_{m\in \mN}$, where $\varphi_m=m\varphi(m\cdot)$, is
also an approximate identity.  Also,
$$
\frac{1}{2} \bbH_{\frac{1}{2}} \varphi_m
= \frac{m}{2} \varphi\left(\frac{m}{2}\cdot\right)
=\frac{m}{2}\displaystyle 2 (\bbH_{2} \psi)\left(\frac{m}{2}\cdot\right)
=m\psi\left(2 \frac{m}{2}\cdot\right)=m\psi(m\cdot)=\psi_m.
$$
Thus,
$$
 \sum_{n\in \mZ} c_n(\bbH_2 T)
=\lim_{m\rightarrow \infty} \langle \bbH_2 T,\varphi_m \rangle
=\lim_{m\rightarrow \infty} \frac{1}{2}\langle T,\bbH_{\frac{1}{2}} \varphi_m\rangle 
=\lim_{m\rightarrow \infty}\langle T,\psi_m\rangle.
$$
Consequently, $T \in \bsigma$, and again
\eqref{eq_16_jan_2020_19_07} holds.
\end{proof}

\noindent Let $T\in \calD'(\mR)$ be such that there
exists a $k\neq 1$ such that $T-k \bbH_2 T$ has summable Fourier
coefficients at $t=0$ in the sense of Definition~\ref{def_summable},
that is,
$$
{\small \sum_{n\in \mZ} c_n(T-k \bbH_2 T)}
$$
is summable. Then we have two possible cases:
\begin{itemize}
\item[$\underline{1}^\circ$] $T\in \bsigma$. Then it follows from
  Propositions~\ref{prop_linearity} and \ref{prop_hom} that
 $$
 {\small
 \sum_{n\in \mZ} c_n(T)= \frac{1}{1-k}\sum_{n\in \mZ} c_n(T-k \bbH_2 T).}
 $$
\item[$\underline{2}^\circ$]\label{page_ref_2_circ}
  $T\not\in \bsigma$, that is, $T$ has a Fourier series that is not
  summable at $t=0$ in the sense of
  Definition~\ref{def_summable}. Nevertheless, in light of $1^\circ$
  above, it is now natural to define
 $$
 {\small \sum_{n\in \mZ} c_n(T)= \frac{1}{1-k}\sum_{n\in \mZ} c_n(T-k \bbH_2 T).}
 $$
 In order to make this a formal definition, we need to check
 well-definedness, which is done below.
\end{itemize}

\begin{proposition}
\label{prop_3_feb_2020_1402}
  Suppose that $T\in \calD'(\mR)$ is a $2\pi$-periodic distribution, 
  such that $T\not \in \bsigma$, but there exist
  $\alpha,\beta \in \mC\setminus \{1\}$ such that
  $ T-\alpha \bbH_2 T\in \bsigma$ and $T-\beta \bbH_2 T\in \bsigma.  $
  Then $\alpha=\beta$.
\end{proposition}
\begin{proof}
  Suppose that $\alpha\neq \beta$. As $T-\alpha \bbH_2 T\in \bsigma$
  and $T-\beta \bbH_2 T\in \bsigma $, we have that for any approximate
  identity $(\varphi_m)_{m\in \mN}$, both the limits
 \begin{eqnarray}
 \label{alpha}
  \sigma_\alpha&:=&\lim_{m\rightarrow \infty} (\langle T,\varphi_m \rangle-\alpha\langle \bbH_2 T,\varphi \rangle) ,\\
  \label{bbb_bets}
  \sigma_\beta&:=&\lim_{m\rightarrow \infty} (\langle T,\varphi_m \rangle-\beta\langle \bbH_2 T,\varphi \rangle) 
 \end{eqnarray}
 exist. But then taking the difference of \eqref{alpha} and \eqref{bbb_bets} gives
 $$
 \sigma_\alpha-\sigma_\beta
 = \lim_{m\rightarrow \infty} (\beta-\alpha)\langle \bbH_2 T,\varphi \rangle,
 $$
 and as $\alpha\neq\beta$, 
 $$
 \lim_{m\rightarrow \infty} \langle \bbH_2 T,\varphi \rangle
 =\frac{1}{\beta-\alpha}(\sigma_\alpha-\sigma_\beta)
 $$
 exists. But then, adding $\alpha$ times
 $\displaystyle \lim_{m\rightarrow \infty} \langle \bbH_2 T,\varphi
 \rangle$ to \eqref{alpha}, gives
 $$
 \lim_{m\rightarrow \infty} \langle T,\varphi_m \rangle
 =\lim_{m\rightarrow \infty} (\langle T,\varphi_m \rangle-\alpha\langle \bbH_2 T,\varphi \rangle)+\alpha 
 \lim_{m\rightarrow \infty} \langle \bbH_2 T,\varphi \rangle
 $$
 exists, that is, $T\in \bsigma$, a contradiction. Hence
 $\alpha=\beta$.
\end{proof}

\noindent Proposition~\ref{prop_3_feb_2020_1402} takes care of the well-definition in
$2^\circ$ on page \pageref{page_ref_2_circ}, and we now have the
following extension of the summation method for Fourier coefficients
given in Definition \ref{def_summable}, from $\bsigma$ to the larger
set $\bSigma$, defined below. 

\begin{notation}[$\bSigma$]$\;$ 

\noindent 
 We set $
\bSigma:=\{T\in \calD'(\mR): \exists k\neq 1 \textrm{ such that } T-k \bbH_2 T \in \bsigma\}$.
\end{notation}

\begin{definition}\label{def_last_summable}
  Let $T\in \calD'(\mR)$ be a $2\pi$-periodic distribution $T$ such
  that there exists a $k\in \mC\setminus\{1\}$ such that
  $T-k \bbH_2 T\in \bsigma$. Then we say that {\em the series}
 $$
 \displaystyle \sum_{n\in \mZ} c_n(T)
 $$ 
 {\em is summable}, or the {\em Fourier series of $T$ is summable at
   $t=0$}, and we define
 $$
 \sum_{n\in \mZ} c_n(T)=\frac{1}{1-k}\sum_{n\in \mZ} c_n(T-k\bbH_2 T).
 $$
 We call a double-sided complex sequence $(C_n)_{n\in \mZ}$ {\em
   summable} if there exists a $2\pi$-periodic distribution $T$ such
 that
\begin{itemize}
 \item for all $n\in \mZ$, $C_n=c_n(T)$, and $\phantom{\displaystyle \sum}$
 \item the Fourier series of $T$ is summable at $t=0$. $\phantom{\displaystyle \sum}$
\end{itemize}
\end{definition}
 
 \section{The distribution ${\small {\mathrm P\mathrm f} {\displaystyle \frac{e^{it}}{(1+e^{it})^2}} }
\in \calD'((0,2 \pi))$}
\label{section_3}

\noindent For $0<\epsilon<\pi$, let
  $\Omega_\epsilon:=(0,\pi-\epsilon )\cup (\pi +\epsilon,2\pi)$.
 
 \medskip 
 
 \noindent 
 For $\varphi \in \calD((0,2\pi))$, we define
 \begin{equation}
  \label{18_jan_2020_16_09}
 \left\langle {\mathrm P\mathrm f} {\displaystyle \frac{e^{it}}{(1+e^{it})^2}} ,\varphi \right\rangle 
 :=
 \lim_{\epsilon\searrow 0} 
 \left( \int_{\Omega_\epsilon} \frac{\varphi(t) e^{it}}{(1+e^{it})^2}dt -\frac{\varphi(\pi)}{\tan (\epsilon/2)}\right). 
 \end{equation}
 In the following result, we will give an alternative expression for the right hand side of \eqref{18_jan_2020_16_09}, 
 and show that ${\mathrm P\mathrm f} {\displaystyle \frac{e^{it}}{(1+e^{it})^2}}$ defines a distribution on the open interval $(0,2\pi)$.

\begin{proposition} 
For any $\varphi \in \calD((0,2\pi))$, we have 
 $$
 \left\langle {\mathrm P\mathrm f} {\displaystyle \frac{e^{it}}{(1+e^{it})^2}} ,\varphi \right\rangle 
 = \int_0^{2\pi} 
 \frac{(t-\pi)^2 e^{it}}{(1+e^{it})^2} 
 \int_0^1 (1-\theta) \varphi''(\pi +\theta (t-\pi)) d\theta dt,
 $$
 and  ${\mathrm P\mathrm f} {\displaystyle \frac{e^{it}}{(1+e^{it})^2}} \in \calD'((0,2\pi))$. 
\end{proposition}
\begin{proof} By Taylor's Formula \cite[(e),p.45]{Zui} we have
$$
\varphi(t)= \varphi(\pi)+(t-\pi)\varphi'(\pi) 
           +(t-\pi)^2 \int_0^1 (1-\theta) \varphi''(\pi+\theta(t-\pi)) d\theta.
$$
We have 
\begin{eqnarray*}
 \int_{\Omega_\epsilon}\frac{\varphi(t) e^{it}}{(1+e^{it})^2}dt
 &=&
 \varphi(\pi)\underbrace{\int_{\Omega_\epsilon} \frac{e^{it}}{(1+e^{it})^2}dt}_{A}
 +
 \varphi'(\pi)\underbrace{\int_{\Omega_\epsilon} \frac{(t-\pi)e^{it}}{(1+e^{it})^2}dt}_{B}
 \\
 &&\phantom{a}+
 \underbrace{\int_{\Omega_\epsilon} \frac{(t-\pi)^2e^{it}}{(1+e^{it})^2}\int_0^1(1-\theta)\varphi''(\pi +\theta(t-\pi)) d\theta dt.}_{C}
\end{eqnarray*}
We will treat the three summands $A,B,C$ separately below.  It can be seen that the second summand 
above vanishes, by splitting the integral $B$ on
$\Omega_\epsilon=(0,\pi-\epsilon)\cup (\pi+\epsilon,2\pi)$ into two,
and using the substitution $\tau=2\pi -t$ for the integral on the
domain $(\pi+\epsilon,2\pi)$:
\begin{eqnarray*}
B&=&\int_{\Omega_\epsilon} \frac{(t-\pi)e^{it}}{(1+e^{it})^2}dt
\\
&=&
\int_{0}^{\pi-\epsilon} \frac{(t-\pi) e^{it}}{(1+e^{it})^2}dt
+
\int_{\pi+\epsilon}^{2\pi} \frac{(t-\pi)e^{it}}{(1+e^{it})^2}dt
\\
&=& 
\int_{0}^{\pi-\epsilon} \frac{(t-\pi)e^{it}}{(1+e^{it})^2}dt
+
\int_{\pi-\epsilon}^{0} \frac{(\pi-\tau) e^{-i\tau}}{(1+e^{-i\tau})^2}(-1) d\tau 
\\
&=&
\int_{0}^{\pi-\epsilon} \frac{(t-\pi)e^{it}}{(1+e^{it})^2}dt
-
\int_{0}^{\pi-\epsilon} \frac{(\tau-\pi) e^{-i\tau}}{(e^{i\tau}+1)^2e^{-2i \tau}} d\tau\\
&=& 
\int_{0}^{\pi-\epsilon} \frac{(t-\pi)e^{it}}{(1+e^{it})^2}dt
-
\int_{0}^{\pi-\epsilon} \frac{(\tau-\pi)e^{i\tau}}{(1+e^{i\tau})^2}d\tau\\
&=&0.\phantom{\int_{0}^{\pi} \frac{e^{i\tau}}{(e^{i\tau})^2}d\tau}
\end{eqnarray*}
In order to compute $A$, we note that $\displaystyle \frac{d}{dt}\frac{1}{1+e^{it}}
=-\frac{ie^{it}}{(1+e^{it})^2}$, and so  
\begin{eqnarray*}
 A=\int_{\Omega_\epsilon} \frac{e^{it}}{(1+e^{it})^2}dt
 &=&
 \frac{i}{1+e^{it}}\Big|_0^{\pi-\epsilon}+\frac{i}{1+e^{it}}\Big|_{\pi+\epsilon}^{2\pi}
 \\
 &=& 
 i\Big(\frac{1}{1+e^{i(\pi-\epsilon)}}-\frac{1}{2}+\frac{1}{2}-\frac{1}{1+e^{i(\pi+\epsilon)}}\Big)
 \\
 &=& 
 i \Big(\frac{e^{i\epsilon}}{e^{i\epsilon}-1}-\frac{1}{1-e^{i\epsilon}}\Big)\\
 &=&i\frac{2\cos (\epsilon/2)}{2i \sin (\epsilon/2)} 
 \\
 &=&\frac{1}{\tan(\epsilon/2)}.
\end{eqnarray*}
Finally we consider the integral $C$, and we are interested in taking
the limit of this integral as $\epsilon\searrow 0$.

\goodbreak

\noindent 
For $t\in [0,2\pi]$, $|\cos(t/2)|\geq \displaystyle \frac{|\pi-t|}{\pi}$.
\begin{figure}[H]
     \center 
     \includegraphics[width=5.7 cm]{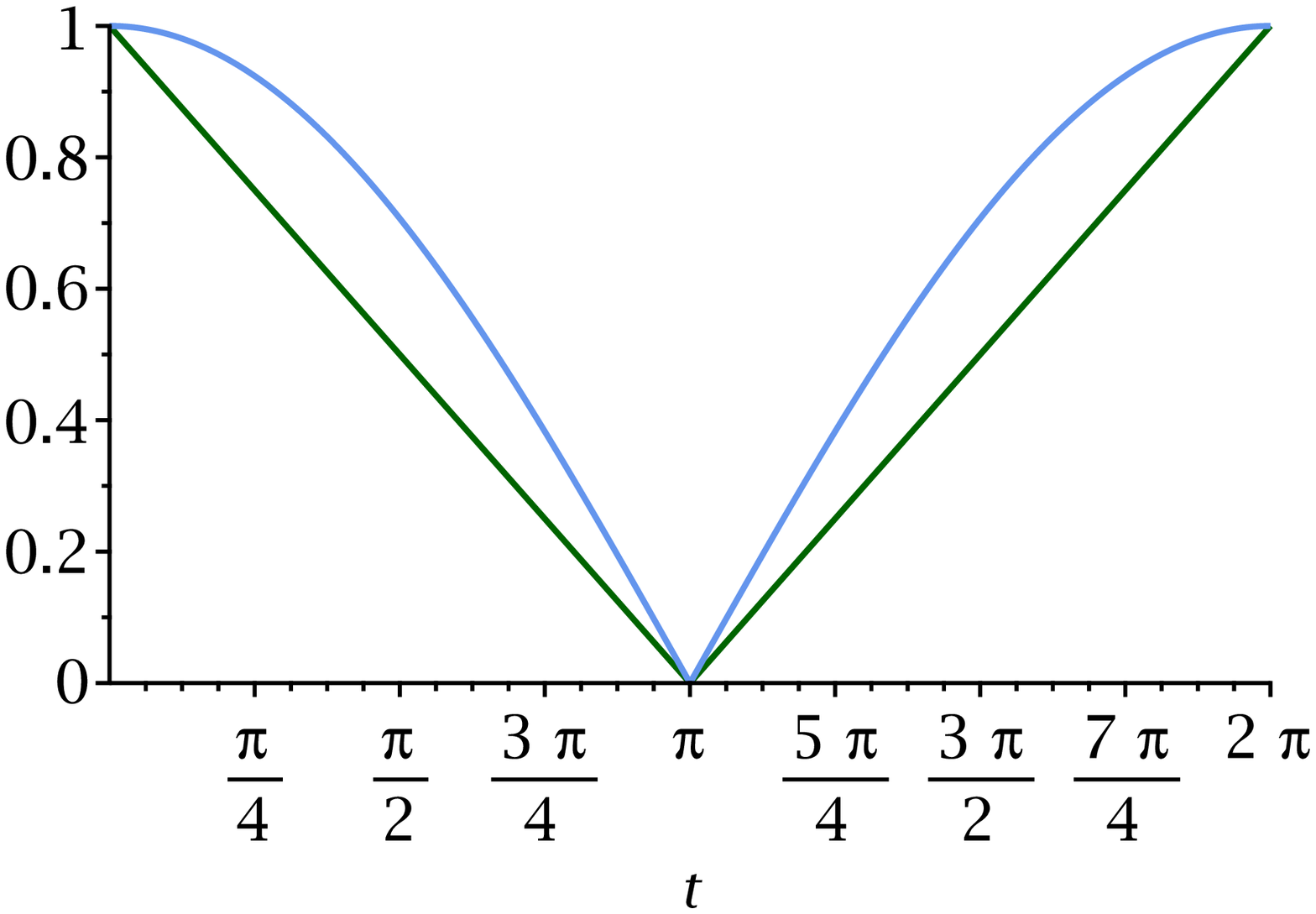}
\end{figure}

\noindent Thus
$$
\displaystyle 
\left| \frac{(t-\pi)^2 e^{it}}{(1+e^{it})^2}\right| 
=
\frac{(t-\pi)^2}{|e^{it/2}+e^{-it/2}|^2} 
=\frac{(t-\pi)^2}{4(\cos (t/2))^2}
\leq \frac{\pi^2}{4}.
$$
So we have 
\begin{eqnarray*}
\lim_{\epsilon\searrow 0} \int_{\Omega_\epsilon} 
&&\!\!\!\!\!\!\!\!\!\!\!\!\!\!\!\frac{(t-\pi)^2e^{it}}{(1+e^{it})^2}
\int_0^1(1-\theta)\varphi''(\pi +\theta(t-\pi)) d\theta dt
\\
&=&
\int_0^{2\pi} 
 \frac{(t-\pi)^2 e^{it}}{(1+e^{it})^2} 
 \int_0^1 (1-\theta) \varphi''(\pi +\theta (t-\pi)) d\theta dt .
\end{eqnarray*}
Hence 
\begin{eqnarray*}
\left\langle {\mathrm P\mathrm f} {\displaystyle \frac{e^{it}}{(1+e^{it})^2}} ,\varphi \right\rangle
&:=& 
\lim_{\epsilon \searrow 0} 
\left( 
\int_{\Omega_\epsilon} \frac{\varphi(t) e^{it}}{(1+e^{it})^2}dt -\frac{\varphi(\pi)}{\tan (\epsilon/2)}\right)
 \\
 &=& \int_0^{2\pi} 
 \frac{(t-\pi)^2 e^{it}}{(1+e^{it})^2} 
 \int_0^1 (1-\theta) \varphi''(\pi +\theta (t-\pi)) d\theta dt .
\end{eqnarray*}
The linearity of
 $$
 \varphi\mapsto \left\langle {\mathrm P\mathrm f} {\displaystyle
     \frac{e^{it}}{(1+e^{it})^2}} ,\varphi \right\rangle
 :\calD((0,\pi))\rightarrow \mC
 $$
is obvious. In order to show continuity, we note that if 
$(\varphi_n)_{n\in \mN}$ is a sequence that converges to $0$ in
$\calD((0,2\pi))$, then in particular, $(\varphi_n'')_{n\in \mN}$
converges to $0$ uniformly on $(0,2\pi)$, and so
\begin{eqnarray*}
\left| 
\left\langle {\mathrm P\mathrm f} {\displaystyle \frac{e^{it}}{(1+e^{it})^2}} ,\varphi_n \right\rangle 
\right|
&=&
\left| 
\int_0^{2\pi} 
 \frac{(t-\pi)^2 e^{it}}{(1+e^{it})^2} 
 \int_0^1 (1-\theta) \varphi_n''(\pi +\theta (t-\pi)) d\theta dt
 \right|
 \\
 &\leq & 2\pi \cdot \frac{\pi^2}{4}\cdot 1 \cdot \|\varphi_n''\|_\infty \stackrel{n\rightarrow \infty}{\longrightarrow} 0.
 \phantom{{\displaystyle \frac{e^{it}}{(e^{it})^2}}}
\end{eqnarray*}
Consequently, ${\mathrm P\mathrm f} {\displaystyle \frac{e^{it}}{(1+e^{it})^2}} 
\in \calD'((0,2\pi))$. 
\end{proof}

\section{$2\pi$-periodic extension}
\label{section_4}

\noindent For $\varphi \in \calD(\mR)$, we define the $2\pi$-periodic
extension of
$$
{\mathrm P\mathrm f} {\displaystyle \frac{e^{it}}{(1+e^{it})^2}} \in \calD'((0,2\pi))
$$ 
to $\mR$, which we will denote 
denoted by the same symbol.
%
%
%
This can be done as follows: We note that as the function
$$
\frac{e^{it}}{(1+e^{it})^2}
$$
is continuous around $0$ and $2\pi$, and there is no problem
extending our old distribution
$$
\displaystyle {\mathrm P\mathrm f} {\displaystyle \frac{e^{it}}{(1+e^{it})^2}} \in \calD'((0,2\pi))
$$ 
to a slightly bigger open set $(0-\delta, 2\pi+\delta)$ for a small
$\delta>0$, so that in fact
$$
\displaystyle {\mathrm P\mathrm f} {\displaystyle \frac{e^{it}}{(1+e^{it})^2}} 
\in \calD'((0-\delta,2\pi+\delta)).
$$ 
Then the patching up of shifts of this distribution by integer
multiples of $2\pi$ in order to define a global distribution on $\mR$
can be done using the following principle of piecewise pasting/`recollement des
morceaux' \cite[Theorem IV, p.27]{Sch1} or \cite[Theorem 1.4.3,
p.16-17]{Kes}.

\begin{proposition}
Let $\{\Omega_i\}_{i\in I}$ be an open cover of $\mR$. Let $T_i\in \calD'(\Omega_i)$ be such that whenever for 
$i\neq j$ we have $\Omega_i\cap \Omega_j\neq \emptyset$, then $T_i|_{\Omega_i\cap \Omega_j}=T_j|_{\Omega_i\cap \Omega_j}$. 
Then there exists a unique distribution $T\in \calD'(\mR)$ such that $T|_{U_i}=T_i$ for all $i\in I$, given by 
$$
\langle T,\varphi\rangle = \sum_{i\in I} \langle T_i ,\varphi \alpha_i\rangle,
$$
where $\{\alpha_i\}_{i\in I}$ is a locally finite partition of unity\footnote{That is, $\alpha_i$, $i\in I$, are $C^\infty$ functions such that 
$0\leq \alpha_i\leq 1$, $E_i:=\textrm{supp}(\alpha_i)\subset \Omega_i$, $\{E_i\}_{i\in I}$ is locally finite (that is, for all $x\in \mR$, there exists 
a neighbourhood of $x$ that intersects only finitely many of the $E_i$, and $\displaystyle\sum\limits_{i\in I}\alpha_i=1$.}.
\end{proposition}

\begin{notation}[The distribution $S$] $\;$
\label{19_jan_2020_distr_S}

\smallskip 

\noindent 
We define the distribution $S\in \calD'(\mR)$ by  
$$
S:= {\mathrm P\mathrm f} {\displaystyle \frac{e^{it}}{(1+e^{it})^2}} 
+i\pi \sum_{n\in \mZ} \delta'_{(2n+1)\pi}.
$$
It follows from the construction of $S$ that $S$ is $2\pi$-periodic.
\end{notation}

\section{Fourier coefficients of 
$S={\mathrm P\mathrm f} {\displaystyle \frac{e^{it}}{(1+e^{it})^2}} 
+i\pi \displaystyle \sum_{n\in \mZ} \delta'_{(2n+1)\pi}$}
\label{sect_FC}

\begin{theorem}
The following Fourier series expansion is valid in $\calD'(\mR)$:
\begin{eqnarray*}
S={\mathrm P\mathrm f} {\displaystyle \frac{e^{it}}{(1+e^{it})^2}} 
+i\pi \displaystyle\! \sum_{n\in \mZ} \delta'_{(2n+1)\pi}
 \!&=&\!e^{it}-2e^{2it}+3e^{3it}-4e^{4it}+-\cdots\\
 \!&=&\!\sum_{n=1}^\infty (-1)^{n-1} n e^{nit}.
\end{eqnarray*}
\end{theorem}
\begin{proof} We produce, for each $\delta>0$, a test function
  $\rho_\delta \in \calD(\mR)$ such that
$$
\rho_{\delta}\Big|_{(\delta,2\pi-\delta)}=1,
$$
and such that we obtain a partition of unity on $\mR$ by taking shifts
of $\rho_\delta$ by integer multiples of $2\pi$, that is,
$$
\sum_{n\in \mZ} \rho_\delta (t+2\pi n)=1\quad (t\in \mR).
$$
\begin{figure}[h]
     \center 
     \psfrag{d}[c][c]{${\scriptstyle \delta}$}
     \psfrag{f}[c][c]{$\varphi$}
     \psfrag{F}[c][c]{$\Phi$}
     \psfrag{p}[c][c]{${\scriptstyle 2\pi}$}
     \psfrag{r}[c][c]{$\rho_\delta$}
     \psfrag{0}[c][c]{${\scriptstyle 0}$}
     \psfrag{t}[c][c]{${\scriptstyle t}$}
     \psfrag{1}[c][c]{${\scriptstyle 1}$}
     \psfrag{s}[c][c]{${\scriptstyle 2\pi-\delta}$}
     \psfrag{m}[c][c]{${\scriptstyle -\delta}$}
     \psfrag{u}[c][c]{${\scriptstyle 2\pi+\delta}$}
     \includegraphics[width=6 cm]{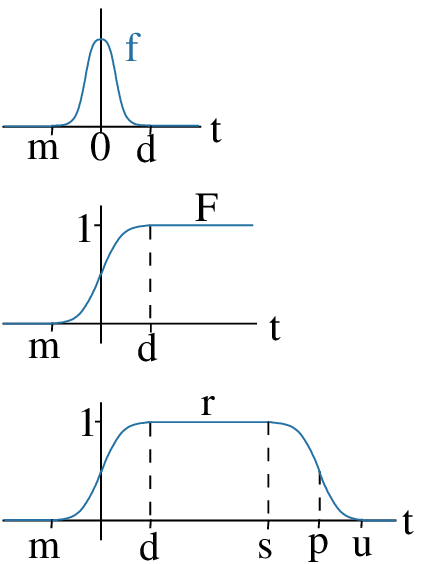}
\end{figure}

\noindent Start with a $\delta>0$ small, and consider any even,
nonnegative test function $\varphi$ with support in $[-\delta,\delta]$
and such that
$$
\int_{-\delta}^{\delta} \varphi(t)dt=1.
$$
Define the function 
$$
\Phi(t):=\int_{(-\infty,t]} \varphi(\tau) d\tau .
$$
Then it can be seen that for all $t\in \mR$, $\Phi$ satisfies
$\Phi(t)+\Phi(-t)=1$.  

\noindent The desired $\rho_\delta$ can now be defined by
$$
\rho_\delta (t)=\Phi(t)\cdot \Phi(2\pi-t).
$$
Define $\varphi_\delta \in \calD(\mR)$ by 
$$
\varphi_\delta (t)=\rho_\delta(t) e^{-int}.
$$
Then we have 
$$
\varphi_{\delta,{ \scriptscriptstyle\textrm{circle}}}(e^{it})
=\sum_{m\in \mZ} \varphi_\delta(t-2\pi m) 
=\sum_{m\in \mZ} \rho_\delta(t-2\pi m) e^{-int}
=e^{-int}.
$$
Hence 
$$
c_n(S)
=\frac{1}{2\pi}\langle S_{\scriptscriptstyle\textrm{circle}},e^{-int}\rangle
=
\frac{1}{2\pi}\langle S_{ \scriptscriptstyle\textrm{circle}} ,
\varphi_{\delta,{\scriptscriptstyle\textrm{circle}}} \rangle=
\frac{1}{2\pi}\langle S,\varphi_{\delta} \rangle.
$$
Before proceeding, we explain the idea in the calculation below: We
note that as $\delta\searrow 0$, $\varphi_\delta$ converges pointwise
to $1$ on $(0,2\pi)$, and to $0$ on $\mR\setminus [0,2\pi]$. So we
expect $\langle S,\varphi_\delta\rangle$ to essentially be the action
of $S$ on $e^{-int}$.  However, to show this, we need to work with the
two limiting processes $\delta\searrow 0$ as well as
$\epsilon\searrow 0$ (arising from the definition of $S$ involving the
distribution ${\mathrm P\mathrm f}(\cdots)$). So we will split the set
containing the support of $\varphi_\delta$ into two parts, one where
the $\epsilon$ limit is relevant, but where $\rho_\delta$ is `nice',
namely a constant equal to $1$, and the part where the $\delta$ limit
is relevant, but the distribution is `nice', namely it is a regular
distribution\footnote{Throughout the article, by a regular distribution $T_f\in \calD'(\mR)$ corresponding to a locally integrable function $f\in \textrm{L}^1_{\textrm{loc}}(\mR)$, 
we mean the distribution $\calD(\mR)\owns \varphi\mapsto \int_\mR f(t) \varphi(t) dt \in \mR$.}.

\smallskip

\noindent We partition $\displaystyle O_\epsilon
=\left(-\frac{\pi}{2},\pi-\epsilon\right)\cup \left(\pi+\epsilon,2\pi+\frac{\pi}{2}\right)$ as  
$$
O_\epsilon=\underbrace{\left(-\frac{\pi}{2},\frac{\pi}{2}\right)\cup\left(3\frac{\pi}{2},2\pi+\frac{\pi}{2}\right)}_{V}\cup 
\underbrace{\left(\frac{\pi}{2},\pi-\epsilon \right)\cup\left(\pi+\epsilon ,3\frac{\pi}{2}\right)}_{V_\epsilon}.
$$
Then we have 
\begin{eqnarray*}
&& \left\langle {\mathrm P\mathrm f} {\displaystyle \frac{e^{it}}{(1+e^{it})^2}} ,\varphi_\delta \right\rangle 
 \\
 &&\phantom{aaa}= 
 \lim_{\epsilon \searrow 0} 
\left( 
\int_{O_\epsilon} \frac{\varphi_\delta(t) e^{it}}{(1+e^{it})^2}dt -\frac{e^{-in\pi}\cdot \rho_\delta(\pi)}{\tan (\epsilon/2)}\right)\\
&&\phantom{aaa}= 
\int_{V} \frac{\varphi_\delta (t) e^{it}}{(1+e^{it})^2}dt +\lim_{\epsilon \searrow 0} \left( \int_{V_\epsilon }\frac{1\cdot e^{-int} e^{it}}{(1+e^{it})^2}dt 
-\frac{(-1)^n}{\tan (\epsilon/2)}\right)
\\
&&\phantom{aaa}= 
\lim_{\delta\searrow 0}
\int_{V} \frac{\rho_\delta (t) e^{-int}e^{it}}{(1+e^{it})^2}dt 
+\lim_{\epsilon \searrow 0} \left( \int_{V_\epsilon }\frac{e^{-int} e^{it}}{(1+e^{it})^2}dt 
-\frac{(-1)^n}{\tan (\epsilon/2)}\right)
\\
&&\phantom{aaa}= 
\int_{V} \frac{1\cdot e^{-int}e^{it}}{(1+e^{it})^2}dt 
+\lim_{\epsilon \searrow 0} \left( \int_{V_\epsilon }\frac{e^{-int} e^{it}}{(1+e^{it})^2}dt 
-\frac{(-1)^n}{\tan (\epsilon/2)}\right)
\\
&&\phantom{aaa}= 
\lim_{\epsilon \searrow 0} 
\left( 
\int_{\Omega_\epsilon} \frac{e^{-int} e^{it}}{(1+e^{it})^2}dt 
-\frac{(-1)^n}{\tan (\epsilon/2)}\right).
\end{eqnarray*}
In the above, we used the Lebesgue dominated convergence theorem in
the evaluation of the $\delta$-limit in order to get the second-last
equality.  Also, 
\begin{eqnarray*}
\left\langle i\pi \sum_{n\in \mZ} \delta'_{(2n+1)\pi},\varphi_\delta \right\rangle 
&=&
-i\pi \varphi_\delta'(\pi)
\\
&=&
-i\pi (-in)e^{-in\pi}\cdot 1
\\
&=&
-(-1)^n n\pi.\phantom{\displaystyle \left\langle i\pi \sum_{n\in \mZ}  \right\rangle }
\end{eqnarray*}
Hence 
\begin{eqnarray*}
 c_n(S)&=& \frac{1}{2\pi}\langle S,\varphi_{\delta} \rangle\phantom{\left\langle {\mathrm P\mathrm f} {\displaystyle \frac{e^{it}}{(e^{it})^2}}\right\rangle }\\
 &=&\left\langle {\mathrm P\mathrm f} {\displaystyle \frac{e^{it}}{(1+e^{it})^2}} 
 +i\pi \sum_{n\in \mZ} \delta'_{(2n+1)\pi} ,
 \varphi_\delta \right\rangle \\
 &=& 
\frac{1}{2\pi} \lim_{\epsilon \searrow 0} 
\left( 
\int_{\Omega_\epsilon} \frac{e^{-int} e^{it}}{(1+e^{it})^2}dt -\frac{(-1)^n}{\tan (\epsilon/2)}-(-1)^n n\pi\right).
\end{eqnarray*}

\medskip 

\noindent 
Suppose that $\epsilon\in (0,\pi)$ is fixed. Let $C_\epsilon$ be the path along
the circular arc $t\mapsto e^{it}$ for
$t\in [\pi+\epsilon,2\pi]\cup[0,\pi-\epsilon]$, traversed in the
anticlockwise direction. The straight line segment $L_\epsilon$ is given by 
$t\mapsto -\cos \epsilon-it$, for
$t\in [-\sin \epsilon,\sin \epsilon]$. Moreover, let $C$ be the
circular path with center $0$ and radius $1/2$, traversed once in an
anticlockwise manner.
\begin{figure}[h]
     \center 
     \psfrag{L}[c][c]{$L_\epsilon$}
     \psfrag{c}[c][c]{$C_\epsilon$}
     \psfrag{C}[c][c]{$C$}
     \includegraphics[width=3.6 cm]{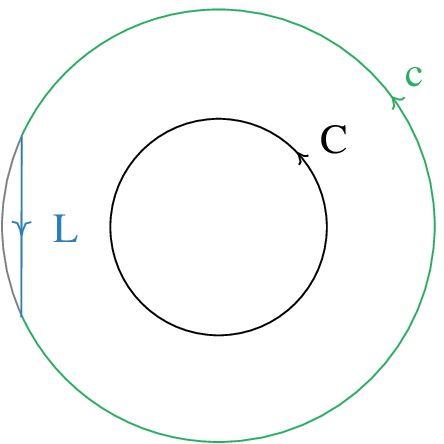}
\end{figure}

\noindent 
We have 
\begin{eqnarray*}
 \frac{1}{2\pi} \int_{\Omega_\epsilon} \frac{e^{-int}e^{it}}{(1+e^{it})^2} dt 
 &=& 
 \frac{1}{2\pi i} \int_{\Omega_\epsilon} \frac{e^{-int}}{(1+e^{it})^2} ie^{it}dt
\\
&=&
 \frac{1}{2\pi i} \int_{C_\epsilon} \frac{z^{-n}}{(1+z)^2} dz\\
 &=& 
 \frac{1}{2\pi i} \int_{C_\epsilon+L_\epsilon} \frac{z^{-n}}{(1+z)^2} dz-\frac{1}{2\pi i} \int_{L_\epsilon} \frac{z^{-n}}{(1+z)^2} dz
 \\
 &=&
 \frac{1}{2\pi i} \oint_{C} \frac{z^{-n}}{(1+z)^2} dz-\frac{1}{2\pi i} \int_{L_\epsilon} \frac{z^{-n}}{(1+z)^2} dz.
\end{eqnarray*}
But by termwise differentiating the geometric series in the disk
$|z|<1$, we obtain
$$
\frac{1}{(1+z)^2}=1-2z+3z^2-4z^3+-\cdots \quad (|z|<1),
$$
and so 
$$
\frac{z^{-n}}{(1+z)^2}=\frac{1}{z^n}(1-2z+3z^2-4z^3+-\cdots+(-1)^{n-1}nz^{n-1}+\cdots \quad (|z|<1),
$$
which yields, by the Laurent series theorem \cite[Thm.4.7,p.134]{Sas}, 
that
$$
\frac{1}{2\pi i} \oint_{C} \frac{z^{-n}}{(1+z)^2} dz
=(-1)^{n-1}n =:d_n \textrm{ if } n=1,2,3,\cdots. 
$$                 
Define 
\begin{eqnarray*}
 \Delta&:=&c_n-d_n\phantom{\displaystyle \frac{1}{2\pi}}\\
 &=&\displaystyle \frac{1}{2\pi} \lim_{\epsilon\searrow 0}
 \left( \int_{L_\epsilon} \frac{z^{-n}}{(1+z)^2}dz -\frac{(-1)^n}{\tan (\epsilon/2)}-(-1)^n n\pi\right).
\end{eqnarray*}
We have, using the fundamental theorem of contour integration, that 
\begin{eqnarray*}
i\int_{L_\epsilon} \frac{1}{(1+z)^2} dz 
&=&
-i \int_{L_\epsilon} \frac{d}{dz} \frac{1}{1+z} dz\\
&=&
-i\left(\frac{1}{1-e^{i\epsilon}}-\frac{1}{1-e^{-i\epsilon}}\right)\\
&=&i\frac{e^{i\epsilon}+1}{e^{i\epsilon}-1}=\frac{1}{\tan(\epsilon/2)}.
\end{eqnarray*}
So 
\begin{eqnarray*}
 \Delta&=& \displaystyle \frac{1}{2\pi} \lim_{\epsilon\searrow 0}
 \left( \int_{L_\epsilon} \frac{z^{-n}}{(1+z)^2}dz -\int_{L_\epsilon} \frac{(-1)^{n}}{(1+z)^2}dz-(-1)^n n\pi\right)\\
 &=& 
 \displaystyle \frac{1}{2\pi} \lim_{\epsilon\searrow 0}
 \left( \int_{L_\epsilon} \frac{z^{-n}-(-1)^n}{(1+z)^2}dz -(-1)^n n\pi\right).
\end{eqnarray*}
If $n=0$, then $\Delta=0$. If $n\neq 0$, then we have, by a Taylor
expansion of $z^{-1}$ around the point $z=-1$, that
$$
z^{-n}=(-1)^n+(-n)(-1)^{-n-1}(z+1)+h(z) 
$$ 
for $z\in L_\epsilon$,
where $h(z)=O((z+1)^2)$ as $z\rightarrow -1$. Hence by the ML
inequality,
$$
\left|\int_{L_\epsilon} \frac{z^{-n}-(-1)^n}{(1+z)^2}dz-\int_{L_\epsilon} \frac{(-n)(-1)^{-n-1}}{1+z}dz\right|\sim \epsilon\stackrel{\epsilon\rightarrow 0}{\longrightarrow} 0.
$$

\goodbreak 

\noindent 
Hence 
\begin{eqnarray*}
 \Delta\displaystyle 
 &=&\frac{1}{2\pi} \lim_{\epsilon\searrow 0}
 \left( \int_{L_\epsilon} \frac{(-n)(-1)^{-n-1}}{1+z}dz -(-1)^n n\pi\right)
 \\
 &=&
 \frac{1}{2\pi } n(-1)^n \lim_{\epsilon\searrow 0}
 \left(i \int_{L_\epsilon} \frac{1}{1+z}dz -\pi\right).
\end{eqnarray*}
But 
\begin{eqnarray*}
 \int_{L_\epsilon} \frac{1}{1+z}dz
 &=& 
 \int_{\sin \epsilon}^{-\sin \epsilon} \frac{1}{1-\cos \epsilon +it} i dt 
 \\
 &=&
 i
 \int_{\sin \epsilon}^{-\sin \epsilon} \frac{1-\cos \epsilon -it}{(1-\cos \epsilon)^2 +t^2}  dt
 \\
 &=& 
 i
 \int_{\sin \epsilon}^{-\sin \epsilon} \frac{1-\cos \epsilon}{(1-\cos \epsilon)^2 +t^2}  dt+0
 \\
 &=&
 -2i(1-\cos \epsilon)\int_0^{\sin \epsilon} \frac{1}{(1-\cos \epsilon)^2 +t^2}  dt.
 \end{eqnarray*}
 Thus 
 \begin{eqnarray*}
  \int_{L_\epsilon} \frac{1}{1+z}dz
 &=& 
 -2i(1-\cos \epsilon) \frac{1}{1-\cos \epsilon}\tan^{-1} \left(\frac{t}{1-\cos \epsilon}\right)\Big|_0^{\sin\epsilon}\\
 &=&-2i \tan^{-1}\frac{\sin\epsilon}{1-\cos \epsilon} \\
 &=&-2i \tan^{-1}\frac{\cos(\epsilon/2)}{\sin( \epsilon/2)}\\
 &=&
 -2i\left(\frac{\pi}{2}-\frac{\epsilon}{2}\right).
\end{eqnarray*}
So 
\begin{eqnarray*}
\Delta&=&\frac{1}{2\pi } n(-1)^n\left(i(-2i)\left(\frac{\pi}{2}-0\right)-\pi\right)\\
&=&\frac{1}{2\pi } n(-1)^n\left(\pi-\pi\right)\\
&=&0.\phantom{\frac{1}{2\pi}}
\end{eqnarray*}
Consequently, 
$$
{\mathrm P\mathrm f} {\displaystyle \frac{e^{it}}{(1+e^{it})^2}}
 +i\pi \displaystyle \sum_{n\in \mZ} \delta'_{(2n+1)\pi}
 =e^{it}-2e^{2it}+3e^{3it}-4e^{4it}+-\cdots.
$$
where the series on the right hand side converges in $\calD'(\mR)$. 
\end{proof}

\goodbreak

\section{The Fourier series of $S$ is summable at $t=0$}
\label{section_5}

\noindent In this section, we will show that the $2\pi$-distribution $S\in \calD'(\mR)$,  
given in Notation~\ref{19_jan_2020_distr_S} on page \pageref{19_jan_2020_distr_S}, 
 belongs to $\bsigma$, that is, it has a Fourier series which is summable at $t=0$ 
 (in the sense of Definition~\ref{def_summable}). 
 
 Let $(\varphi_m)_{m\in \mN}$ be any approximate
identity. Then we have, for all large enough $m$, that the support of
$\varphi_m$ is contained inside, say $(-\pi/2,\pi/2)$, and so, in
particular, it is far from $\pm\pi$. Thus if
$U_\epsilon:=(-\pi/2, \pi-\epsilon)$, then we have for all large $m$,
that
{\small 
\begin{eqnarray*}
\left\langle \!{\mathrm P\mathrm f} {\displaystyle \frac{e^{it}}{(1\!+\!e^{it})^2}} \!+\!i\pi \displaystyle \sum_{n\in \mZ} \!\delta'_{(2n+1)\pi},\varphi_m \!\right\rangle 
\!\!\!&\!\!\!=&\!\!\!
\lim_{\epsilon\searrow 0} \!
 \left( \int_{U_\epsilon}\! \frac{\varphi_m(t) e^{it}}{(1\!+\!e^{it})^2}dt \!-\!\frac{0}{\tan (\epsilon/2)}\!\right)\!+\!i\pi\! \cdot\! 0\\
\\
&=&
\int_{-\pi/2}^{\pi/2} \frac{\varphi_m(t) e^{it}}{(1+e^{it})^2}dt.
\end{eqnarray*}}
Hence, 
\begin{eqnarray*}
\lim_{m\rightarrow \infty} \langle S, \varphi_m\rangle 
&=&
\lim_{m\rightarrow \infty}
\left\langle {\mathrm P\mathrm f} \displaystyle \frac{e^{it}}{(1+e^{it})^2} +i\pi \sum_{n\in \mZ} \delta'_{(2n+1)\pi},\varphi_m \right\rangle 
\\
&=& 
\lim_{m\rightarrow \infty}
\int_{-\pi/2}^{\pi/2} \frac{\varphi_m(t) e^{it}}{(1+e^{it})^2}dt=\lim_{m\rightarrow \infty}
\int_{-\pi/2}^{\pi/2} \frac{\varphi_m(t) e^{it}\psi(t)}{(1+e^{it})^2}dt ,
\end{eqnarray*}
where $\psi\in \calD(\mR)$ is such that $\psi\equiv 1$ on $(-\pi/2,\pi/2)$ and has compact support contained inside $(-\pi,\pi)$. 
Thus 
\begin{eqnarray*}
\lim_{m\rightarrow \infty} \langle S, \varphi_m\rangle 
&=& 
\label{17_jan_2020_1819}
\lim_{m\rightarrow \infty}
\left\langle \varphi_m, \frac{ e^{it}\psi }{(1+e^{it})^2}\right\rangle \quad\phantom{aaaa} (\ast)\\
&=& 
\frac{1\cdot 1}{(1+1)^2}\quad\phantom{aaaaaaaaaaaaaaall} (\star) \\
&=&
\frac{1}{4}.
\end{eqnarray*}
In the first equality ($\ast$) above, we view $\varphi_m$ as the
regular distribution in $\calE'(\mR)$ corresponding to $\varphi_m$, and
$$
t\mapsto \frac{e^{it}\psi(t)}{(1+e^{it})^2 }
$$
as a test function in $\calE(\mR)$. Then, since 
$$
\lim_{m\rightarrow \infty}\varphi_m=\delta_0 
$$
in $\calD'(\mR)$, the second equality $(\star$) above follows.

\smallskip 

\noindent 
Consequently, $S\in \bsigma$, that is, the distribution 
 $
S= \displaystyle {\mathrm P\mathrm f} {\displaystyle \frac{e^{it}}{(1+e^{it})^2}} +i\pi \displaystyle \sum_{n\in \mZ} \delta'_{(2n+1)\pi}
$ 
has a Fourier series that is summable at $t=0$, and 
$$
1-2+3-4+\cdots=\sum_{n\in \mZ} c_n(S)= \frac{1}{4}.
$$

\section{Ramanujan manipulation using Fourier series}
\label{section_6}

\noindent 
Let $T_0\in \calD'(\mR)$ be the $2\pi$-periodic distribution having
the Fourier series
\begin{equation}
\label{def_distr}
T_0:= e^{it}+2e^{2it}+3e^{3it}+\cdots.
\end{equation}
The series converges in $\calD'(\mR)$. We show in this section that $T_0\in \bSigma$, that is,  
that $T_0$ has a Fourier series which is summable at $t=0$ in the sense of Definition~\ref{def_last_summable}, 
with the sum $-1/12$. 

\medskip

As $T\mapsto \bbH_\lambda T:\calD'(\mR)\rightarrow \calD'(\mR)$ is
continuous, it follows from \eqref{def_distr} that
\begin{equation}
\label{def_distr_2}
\bbH_2 T_0=e^{2it}+2e^{4it}+3e^{6it}+4e^{8it}+5e^{10it}+\cdots.
\end{equation}
Moreover,
\begin{equation}
\label{def_distr_3}
4 \bbH_2 T_0=4e^{2it}+8e^{4it}+12e^{6it}+16e^{8it}+20e^{10it}+\cdots.
\end{equation}
Subtracting \eqref{def_distr} and \eqref{def_distr_3}, we obtain 
\begin{eqnarray*}
T_0-4\bbH_2 T_0&=&e^{it}+2e^{2it}+3e^{3it}+4e^{4it}+5e^{5it}+6e^{6it}+\cdots\\
&& \phantom{e^{it}}-4e^{2it}\phantom{+3e^{6it}a}-8e^{4it}\phantom{+5e^{10it}}-12e^{6it}-\cdots \\
&=& e^{it}-2e^{2it}+3e^{3it}-4e^{4it}+5e^{5it}-6e^{6it}+-\cdots.
\end{eqnarray*}
Thus we arrive at the identity
\begin{equation}\label{equation_main}
T_0-4\bbH_2 T_0 =
{\mathrm P\mathrm f} {\displaystyle \frac{e^{it}}{(1+e^{it})^2}} +i\pi \displaystyle \sum_{n\in \mZ} \delta'_{(2n+1)\pi}=:S.
\end{equation}
But we know from the previous section that
 $\displaystyle 
\sum_{n\in \mZ} c_n(S)=\frac{1}{4}.
$

\medskip 

\noindent 
Hence it follows from Definition~\ref{def_last_summable} that 
$$
1+2+3+\cdots= \sum_{n\in \mZ} c_n(T_0)=\frac{1}{1-4}\sum_{n\in \mZ} c_n(S)=\frac{1}{-3}\cdot \frac{1}{4}=-\frac{1}{12}.
$$
Here we used the (extended) summability notion from
Definition~\ref{def_last_summable}, and so we have shown that
$T_0\in \bSigma$.  Had $T_0$ belonged to $\bsigma$, we would of course
obtain the same sum of the Fourier coefficients of $T_0$ at $t=0$.
But we now show (in Proposition~\ref{prop_17_jan_2020_1443} below) that in fact $T_0\not\in \bsigma$ (that is, the distribution $T_0$ 
has a Fourier series that is {\em not} summable at $t=0$ {\em in the restricted sense of
Definition~\ref{def_summable}}).

\begin{proposition}
\label{prop_17_jan_2020_1443}
$\displaystyle T_0:=\sum_{n=1}^\infty ne^{nit} \not\in \bsigma.
$
\end{proposition}
\begin{proof} With $\bbT_{\pi}$ denoting the translation operator by $\pi$, which we note is a continuous map linear map from $\calD'(\mR)$ to itself, we have 
\begin{eqnarray*}
\bbT_{\pi} T_0&=& \bbT_{\pi}\sum_{n=1}^\infty ne^{nit}=\sum_{n=1}^\infty\bbT_{\pi} ne^{nit}\\
&=&\sum_{n=1}^\infty ne^{ni(t-\pi)}=\sum_{n=1}^\infty (-1)^n ne^{nit}
=-S,
\end{eqnarray*}
and so 
\begin{equation}
 \label{3_3_2020_18:25}
 T_0=-\bbT_{-\pi}S.
\end{equation}
Using this, we now calculate (for later use below) the action of $T_0$ on test functions of a special type.  
For any $\chi \in \calD(\mR)$ with the properties that $\textrm{supp}(\chi)\subset (-\pi/2,\pi/2)$ and $\chi(0)=\chi'(0)=0$, we have 
\begin{eqnarray*}
 \langle T_0,\chi\rangle \!&\!\!\!=&\!\!\!-\langle \bbT_{-\pi}S,\chi\rangle =-\langle S,\bbT_{\pi}\chi\rangle \phantom{\lim_{\epsilon\searrow} \int_{{\scriptstyle \pi}}}
 \\
 &\!\!\!=& \!\!\!-\lim_{\epsilon\searrow 0} \left( \int_{{\scriptstyle (0,\pi-\epsilon)\cup(\pi+\epsilon,2\pi)}}\frac{e^{it}\chi(t-\pi)}{(1+e^{it})^2} dt 
 -\frac{\chi(\pi-\pi)}{\tan(\epsilon/2)}\right)
 \\
 &\!\!\!=&\!\!\!-\lim_{\epsilon\searrow 0} \left( -\int_{{\scriptstyle (-\pi,-\epsilon)\cup(\epsilon,\pi)}}\frac{e^{i\tau}\chi(\tau)}{(1-e^{i\tau})^2} d\tau  
 -\frac{0}{\tan(\epsilon/2)}\right) 
 \\
 &\!\!\!=&\!\!\!\lim_{\epsilon\searrow 0} \int_{{\scriptstyle (-\pi,-\epsilon)\cup(\epsilon,\pi)}}\frac{e^{i\tau}\chi(\tau)}{(1-e^{i\tau})^2} d\tau 
\end{eqnarray*}
In the above, we used the substitution $\tau=t-\pi$ in order to obtain the equality in the third row.

Summarising, for $\chi \in \calD(\mR)$ with the properties that $\textrm{supp}(\chi)\subset (-\pi/2,\pi/2)$ and $\chi(0)=\chi'(0)=0$, we have 
\begin{equation}
 \label{3_3_2020_18:39}
 \langle T_0,\chi\rangle=
 \lim_{\epsilon\searrow 0} \int_{{\scriptstyle (-\pi,-\epsilon)\cup(\epsilon,\pi)}}\frac{e^{i\tau}\chi(\tau)}{(1-e^{i\tau})^2} d\tau .
\end{equation}
This fact will be used below, where $\chi$ will be replaced by the elements of an approximate identity. 

 Let $\psi$ be a nonzero test function in $\calD(\mR)$ which is symmetric and nonnegative, and has its support 
  $\textrm{supp}(\psi)\subset (-1,1)\subset (-\pi/2,\pi/2)$. Define $\varphi\in \calD(\mR)$ by 
  $$
  \varphi(t)=\frac{t^4\psi(t)}{{ \displaystyle \int_{-1}^1 t^4 \psi(t) dt}}=\frac{1}{I} \cdot t^4 \psi(t) \quad (t\in \mR),\textrm{ where }
  \displaystyle I:= \int_{-1}^1 t^4 \psi(t) dt>0.
  $$
  Then $\varphi$ is a symmetric positive mollifier with $\varphi(0)=\varphi'(0)=0$.
  Analogously, we also define $\widetilde{\varphi}\in \calD(\mR)$ by 
  $$
  \widetilde{\varphi}(t)=\frac{t^2\psi(t)}{{ \displaystyle \int_{-1}^1 t^2 \psi(t) dt}}=\frac{1}{\widetilde{I}}\cdot  t^2 \psi(t) \quad (t\in \mR),\textrm{ where }
  \displaystyle \widetilde{I}:= \int_{-1}^1 t^2 \psi(t) dt>0.
  $$
  Then $\widetilde{\varphi}$ is a symmetric positive mollifier with $\widetilde{\varphi}(0)=\widetilde{\varphi}'(0)=0$.
  
  \smallskip 
  
  \noindent 
  In \eqref{3_3_2020_18:39}, if in particular, we take $\chi=\varphi_m$, 
  then we obtain 
  \begin{eqnarray*}
   \langle T_0,\varphi_m\rangle 
   &=&
   \lim_{\epsilon\searrow 0} \int_{{\scriptstyle (-\pi,-\epsilon)\cup(\epsilon,\pi)}}\frac{e^{it}\varphi_m(t)}{(1-e^{it})^2} dt 
   \\
   &=&
   \frac{1}{I} \cdot \lim_{\epsilon\searrow 0} \int_{{\scriptstyle (-\pi,-\epsilon)\cup(\epsilon,\pi)}}\frac{e^{it}m\cdot m^4 t^4 \psi(mt)}{(1-e^{it})^2} dt 
   \\
   &=&
   m^2 \frac{\widetilde{I}}{I} \cdot \lim_{\epsilon\searrow 0} \int_{{\scriptstyle (-\pi,-\epsilon)\cup(\epsilon,\pi)}}\frac{m\cdot m^2 t^2 \psi(mt)}{\widetilde{I}} 
   \frac{t^2e^{it}}{(1-e^{it})^2} dt \\
   &=&
   m^2 \frac{\widetilde{I}}{I} \cdot\int_{-\pi}^{\pi} \widetilde{\varphi}_m(t)  \frac{t^2e^{it}\beta(t)}{(1-e^{it})^2} dt ,
   \end{eqnarray*}
   where $\beta \in \calD(\mR)$ is such that $\beta\equiv 1$ on $(-\pi,\pi)$ and has compact support contained inside $(-2\pi,2\pi)$. 
   Hence 
   \begin{eqnarray*}
   \langle T_0,\varphi_m\rangle 
   &=&
   m^2 \frac{\widetilde{I}}{I} \cdot \left\langle \widetilde{\varphi}_m, \frac{t^2e^{it}\beta(t)}{(1-e^{it})^2}\right\rangle,
  \end{eqnarray*}
  where in right hand side of the equality, we view $\widetilde{\varphi}_m$ as the
regular distribution in $\calE'(\mR)$ corresponding to $\widetilde{\varphi}_m$, and
$$
t\mapsto \frac{t^2e^{it}\beta(t)}{(1-e^{it})^2 }
$$
as a test function in $\calE(\mR)$. Then, since 
$$
\lim_{m\rightarrow \infty}\widetilde{\varphi}_m=\delta_0 
$$
in $\calD'(\mR)$, it follows that 
\begin{eqnarray*}
\lim_{m\rightarrow \infty} \left\langle \widetilde{\varphi}_m, \frac{t^2e^{it}\beta(t)}{(1-e^{it})^2}\right\rangle
&=&
\left\langle  \delta_0 , \frac{t^2e^{it}\beta(t)}{(1-e^{it})^2}\right\rangle
=\lim_{t\rightarrow 0} 
\frac{t^2e^{it}\beta(t)}{(1-e^{it})^2}
\\
&=&
\lim_{t\rightarrow 0} \frac{e^{it}\beta(t)}{\displaystyle i^2\left(\frac{1-e^{it}}{-it}\right)^2}
=
\frac{e^{i0}\cdot \beta(0)}{i^2\displaystyle \left(\frac{d}{dz} e^{z}\Big|_{z=0}\right)^2}\\
&=& \frac{1\cdot 1}{-1\cdot 1^2}=-1.
\end{eqnarray*}
Consequently, 
$$
\lim_{m\rightarrow \infty} \langle T_0,\varphi_m\rangle = \lim_{m\rightarrow \infty} m^2 \frac{\widetilde{I}}{I} \cdot (-1)=-\infty,
$$
showing that $T_0\not \in \bsigma$. 
\end{proof}

\section{Computation of $1^k+2^k+3^k+\cdots$ for $k\in \mN$} 
\label{section_7}

\noindent We will first show the following result, where $S^{(k)}$ denotes
the $k$th order derivative of the distribution $S$, where $S$ is the distribution  
$$
S:= {\mathrm P\mathrm f} {\displaystyle \frac{e^{it}}{(1+e^{it})^2}} 
+i\pi \sum_{n\in \mZ} \delta'_{(2n+1)\pi}.
$$
It is easy to see,
using the fact that the translation operator commutes with the
differentiation operator for test functions, that $S^{(k)}$ is also
$2\pi$-periodic.

\begin{proposition}
  For all $k\in \mN$, the $2\pi$-periodic distribution $S^{(k)}\in \bSigma$, that is, $S^{(k)}$ has a
  Fourier series that is summable at $t=0$ in the sense of Definition~\ref{def_last_summable}.
\end{proposition}
\begin{proof} Let $(\varphi_m)_{m\in \mN}$ be any approximate
  identity. For all large enough $m$, the support of $\varphi_m$ is
  contained inside $(-\pi/2, \pi/2)$. Then we have
\begin{eqnarray}
\nonumber 
 \lim_{m\rightarrow \infty}
 \langle S^{(k)}, \varphi_m\rangle 
 &=& 
 \lim_{m\rightarrow \infty}
(-1)^k \langle S, \varphi_m^{(k)}\rangle 
\\
\nonumber 
&=&
 \lim_{m\rightarrow \infty}
(-1)^k \int_{-\pi/2}^{\pi/2} \frac{\varphi_m^{(k)}(t) e^{it}}{(1+e^{it})^2}dt
\\
\nonumber 
&=& 
 \lim_{m\rightarrow \infty}
(-1)^k
\left\langle \varphi_m^{(k)}, \frac{ e^{it}}{(1+e^{it})^2}\right\rangle 
\\
\nonumber 
&=&
 \lim_{m\rightarrow \infty}
\left\langle \varphi_m, \left(\frac{d}{dt}\right)^k\frac{ e^{it}}{(1+e^{it})^2}\right\rangle 
\\
\nonumber 
&=& 
\left\langle \delta_0, \left(\frac{d}{dt}\right)^k\frac{ e^{it}}{(1+e^{it})^2}\right\rangle 
\\
\label{17_jan_2020_18_39}
&=& 
\left.\left(\frac{d}{dt}\right)^k\frac{ e^{it}}{(1+e^{it})^2}\right|_{t=0}.
\end{eqnarray} 
This completes the proof.
\end{proof}

\noindent We will also need the commutation relation between the
operations of $\bbH_2$ and differentiation below.

\begin{lemma}
 For any distribution $T\in \calD'(\mR)$, any $\lambda>0$, and any $k\in \mN$,  
 $$
 \left(\bbH_\lambda T\right)^{(k)} =\lambda^k \bbH_\lambda  \left(T^{(k)}\right).
 $$
\end{lemma}
\begin{proof}
 For any test function $\varphi\in \calD(\mR)$, we have 
 \begin{eqnarray*}
  \left\langle \left(\bbH_\lambda T\right)^{(k)} ,\varphi\right\rangle
  &=&
  (-1)^k   \left\langle \bbH_\lambda T ,\varphi^{(k)} \right\rangle\phantom{\displaystyle \frac{(-1)^k}{\lambda}}
  \\
  &=&
  \frac{(-1)^k}{\lambda}  \left\langle  T ,\bbH_{\frac{1}{\lambda}}\left(\varphi^{(k)}\right) \right\rangle
  \\
  &=& 
  \frac{(-1)^k}{\lambda}   \left\langle  T ,\lambda^k \left(\bbH_{\frac{1}{\lambda}}\varphi\right)^{(k)} \right\rangle\phantom{\displaystyle \frac{(-1)^k}{\lambda}}
  \\
  &=&
  \frac{\lambda^k}{\lambda}  \left\langle  T^{(k)} , \bbH_{\frac{1}{\lambda}}\varphi \right\rangle
  \\
  &=& 
  \lambda^k  \left\langle  \bbH_{\lambda} \left(T^{(k)}\right) , \varphi \right\rangle\phantom{\displaystyle \frac{(-1)^k}{\lambda}}
  \\
  &=&
  \left\langle  \lambda^k\cdot  \bbH_{\lambda}  \left(T^{(k)}\right) , \varphi \right\rangle.\phantom{\displaystyle \frac{(-1)^k}{\lambda}}
 \end{eqnarray*}
Hence $\left(\bbH_\lambda T\right)^{(k)} =\lambda^k \bbH_\lambda \left(T^{(k)}\right)$.
\end{proof}

\noindent As $T_0-4\bbH_2 T_0=S$, we obtain by differentiating $k$
times that
$$
T_0^{(k)} -4 \cdot 2^k\cdot \bbH_2 (T_0^{(k)}) =S^{(k)}.
$$
As $S^{(k)} \in \bsigma$, and since $k:=4 \cdot 2^k\neq 1$, we
conclude that the Fourier series of $T_0^{(k)}$ is summable at $t=0$,
and we have
\begin{eqnarray*}
 \sum_{n\in \mZ}c_n(T_0^{(k)})
 &=& 
 \frac{1}{1-4\cdot 2^k} \sum_{n\in \mZ  } c_n(S^{(k)})
 \\
 &=& 
  \frac{1}{1-4\cdot 2^k}\cdot \left.\left(\frac{d}{dt}\right)^k\frac{ e^{it}}{(1+e^{it})^2}\right|_{t=0}.
\end{eqnarray*}
But, because of the convergence of the Fourier series of $T_0$ in
$\calD'(\mR)$, we obtain by termwise differentiation of
$$
T_0=e^{it}+2 e^{2it}+3e^{3it}+\cdots 
$$
that 
$$
T_0^{(k-1)} 
=
i^{k-1} (e^{it}+2^k e^{2it}+3^ke^{3it}+\cdots),
$$
and so we obtain 
\begin{eqnarray}
\nonumber  
1^k+2^k+3^k+\cdots&=&\frac{1}{i^{k-1}} \cdot\frac{1}{1-4\cdot 2^{k-1}}\cdot \left.\left(\frac{d}{dt}\right)^{k-1}\frac{ e^{it}}{(1+e^{it})^2}\right|_{t=0}\\
\label{2_feb_2020_1359}
&=&\frac{1}{i^{k-1}} \cdot\frac{1}{1-2^{k+1}}\cdot \left.\left(\frac{d}{dt}\right)^{k-1}\frac{ e^{it}}{(1+e^{it})^2}\right|_{t=0}.
\end{eqnarray}

%

\noindent 
We note that \eqref{2_feb_2020_1359} gives, for example, 
 \begin{eqnarray*}
 1^2+2^2+3^2+\cdots &=& \frac{1}{i(1-4\cdot 2)} \frac{d}{dt}\frac{e^{it}}{(1+e^{it})^2}\Big|_{t=0}
 \\
 &=& -\frac{ie^{it}}{7i}\left( \frac{1}{(1+e^{it})^2}-\frac{2e^{it}}{(1+e^{it})^3}\right)\Big|_{t=0}\\
 &=& -\frac{i}{7i}\left( \frac{1}{4}-\frac{2}{8}\right)=0.
 \end{eqnarray*}
 As one more example, \eqref{2_feb_2020_1359} gives
 \begin{eqnarray*}
 1^3+2^3+3^3+\cdots &=& \frac{1}{i^2(1-4\cdot 4)} \frac{d^2}{dt^2}\frac{e^{it}}{(1+e^{it})^2}\Big|_{t=0}
 \\
 &=& \frac{1}{15}\left( -\frac{e^{it}}{(1+e^{it})^2}+\frac{6(e^{it})^2}{(1+e^{it})^3}-\frac{6(e^{it})^3}{(1+e^{it})^4} \right)\Big|_{t=0}\\
 &=& \frac{1}{15}\left( \frac{1}{4}+\frac{6}{8}-\frac{6}{16}\right)=\frac{1}{120}.
 \end{eqnarray*} 
We note that these values match the values of the usual analytic
 continuation of the Riemann zeta function $\zeta(s)$ at $s=-2$ and at $s=-3$ respectively; see for example
 \cite[\S8.2]{Sto}.  Here $\zeta$ denotes the Riemann-zeta function defined by 
$$
\zeta(s) :=\sum_{n=1}^\infty n^{-s}.
$$
The series for $\zeta(s)$ converges absolutely if
$\textrm{Re}(s)>1$, but it diverges whenever $\textrm{Re}(s)\leq 1$.
However, one can use an analytic continuation of the zeta function  
 in the punctured plane $\mC\setminus \{1\}$ for determining the zeta function
at points where the series fails to converge. The analytic continuation satisfies a functional equation (see for example \cite[Theorem~2.1]{Tit}), 
which yields  for $k\in \mN$ that 
\begin{equation}
\label{2_feb_2020_14_01}
\zeta(-k)=\frac{2}{(2\pi)^{k+1}} \sin \left(-k \frac{\pi}{2}\right)k! \zeta(1+k);
\end{equation}
We justify in the next two results below that our formula 
 \eqref{2_feb_2020_1359} coincides with $\zeta(-k)$ for all $k\in \mN$.
  Proposition~\ref{prop_arne} is due to Arne Meurman (personal communication). 

\begin{proposition}
\label{prop_arne}
For all odd $k\in \mN$, 
$$
\zeta(-k)=\frac{1}{i^{k-1}} \frac{1}{1-2^{k+1}} \left. \left( \frac{d}{dz}\right)^{k-1} \frac{e^{iz}}{(1+e^{iz})^2} \right|_{z=0}.
$$
\end{proposition}
\begin{proof}  Let $k$ be a positive odd integer and set 
$$
f(z)=\frac{1}{z^{k+1} (1+e^{iz})} .
$$
We shall consider 
$\displaystyle 
\int_{\Gamma_N} f(z) dz, 
$ 
where $\Gamma_N$ denotes the contour (shown below) 
$$
[-N\pi(1+i),-N\pi i+N\pi, N\pi i +N\pi, N\pi i-N\pi, -N\pi(1+i)], 
$$
and $N$ is a positive even integer that shall approach $\infty$. 
\begin{figure}[h]
     \center 
     \psfrag{G}[c][c]{${\scriptstyle \Gamma_N}$}
     \includegraphics[width=4.8 cm]{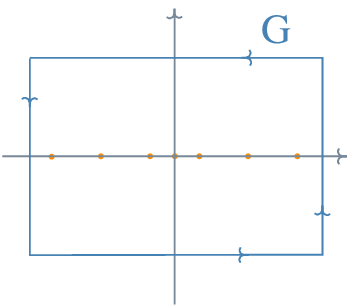}
\end{figure}

\noindent Estimating, one find that $|f(z)|=O(N^{-k-1})$ on $\Gamma_N$, so that 
$$
\lim_{N\rightarrow \infty} \int_{\Gamma_N}f(z) dz=0.
$$
As $f$ has poles at $z=0$ and $z=(2j+1)\pi$, $j\in \mZ$, the residue theorem gives 
\begin{equation}
\label{1_feb_2020_16_29}
0=\textrm{Res}_{z=0} f(z)+\sum_{j\in \mZ} \textrm{Res}_{z=(2j+1)\pi} f(z).
\end{equation}
One obtains 
$$
\textrm{Res}_{z=(2j+1) \pi} f(z)=\frac{1}{[(2j+1)\pi]^{k+1} (-i)} ,
$$
so that 
\begin{eqnarray*}
 \sum_{j\in \mZ} \textrm{Res}_{z=(2j+1) \pi} f(z)
 &=&\frac{2i}{\pi^{k+1}}\sum_{j=0}^\infty  \frac{1}{(2j+1)^{k+1}}
 \\
 &=& 
 \frac{2i}{\pi^{k+1}}\sum_{j=0}^\infty \left( \sum_{n=1}^\infty \frac{1}{n^{k+1}}-\sum_{n=1}^\infty \frac{1}{(2n)^{k+1}}\right)
 \\
 &=&
 \frac{2i}{\pi^{k+1}} \left(1-\frac{1}{2^{k+1}}\right) \zeta(k+1).
\end{eqnarray*}
Moreover, 
\begin{eqnarray*}
 -\textrm{Res}_{z=0} f(z)&=&\left. -\frac{1}{k!} \left(\frac{d}{dz}\right)^k\frac{1}{1+e^{iz}}\right|_{z=0}
 \\
 &=& \left. -\frac{1}{k!} \left(\frac{d}{dz}\right)^{k-1}\frac{(-i) e^{iz}}{(1+e^{iz})^2}\right|_{z=0}
 \\
 &=& \left. \frac{i}{k!} \left(\frac{d}{dz}\right)^{k-1}\frac{ e^{iz}}{(1+e^{iz})^2}\right|_{z=0}.
\end{eqnarray*}
Thus \eqref{1_feb_2020_16_29} gives 
\begin{equation}
\label{1_feb_2020_16_41}
\frac{2i}{\pi^{k+1}} \left(1-\frac{1}{2^{k+1}}\right) \zeta(k+1)=
\left. \frac{i}{k!} \left(\frac{d}{dz}\right)^{k-1}\frac{ e^{iz}}{(1+e^{iz})^2}\right|_{z=0}.
\end{equation}
Recall that the functional equation for $\zeta$ gives 
\begin{equation}
\label{1_feb_2020_16_54}
\zeta(-k)=\frac{2(-1)^{\frac{k+1}{2}} k!}{(2\pi)^{k+1}}\zeta(k+1).
\end{equation}
Comparing \eqref{1_feb_2020_16_41} and \eqref{1_feb_2020_16_54} gives the formula 
$$
\zeta(-k)=\frac{1}{i^{k-1}} \frac{1}{1-2^{k+1}} \left. \left( \frac{d}{dz}\right)^{k-1} \frac{e^{iz}}{(1+e^{iz})^2} \right|_{z=0}.
$$
This completes the proof.
\end{proof}

\begin{proposition}
\label{prop_arne_2}
For all even $k\in \mN$, 
$$
\zeta(-k)=0=\frac{1}{i^{k-1}} \frac{1}{1-2^{k+1}} \left. \left( \frac{d}{dz}\right)^{k-1} \frac{e^{iz}}{(1+e^{iz})^2} \right|_{z=0}.
$$
\end{proposition}
\begin{proof} For even $k\in \mN$, since $\sin\left(-k \frac{\pi}{2}\right)=0$, it follows that $\zeta(-k)=0$.
We will show that in our formula \eqref{2_feb_2020_1359}, also 
$$
\left.\left(\frac{d}{dt}\right)^{k-1}\frac{ e^{it}}{(1+e^{it})^2}\right|_{t=0}=0
$$
for even $k$, so that \eqref{2_feb_2020_1359} matches $\zeta(-k)=0$. Set 
$$
g(z)=\frac{1}{1+e^z}.
$$
Then $g$ is holomorphic in a neighbourhood of $0$. So $h:=g(i\cdot)$ is also holomorphic in a neighbourhood of $0$, and by a 
repeated application of the chain rule, 
$$
\left.\left(\frac{d}{dz}\right)^k h\right|_{z=0}
=
i^k \left.\left(\frac{d}{dz}\right)^k g\right|_{z=i0=0}.
$$
But 
\begin{eqnarray*}
\left.\left(\frac{d}{dz}\right)^k h\right|_{z=0}
&=&
\left.\left(\frac{d}{dz}\right)^{k-1}\frac{d}{dz} \frac{1}{1+e^{iz}} h\right|_{z=0}
=
\left.\left(\frac{d}{dz}\right)^{k-1} \frac{-ie^{iz}}{(1+e^{iz})^2 }\right|_{z=0} 
\\
&=&
-i \left.\left(\frac{d}{dz}\right)^{k-1}\frac{e^{iz}}{(1+e^{iz})^2 }\right|_{z=0} ,
\end{eqnarray*}
and so 
\begin{equation}
 \label{2_feb_2020_14_24} 
 \left.\left(\frac{d}{dz}\right)^{k-1}\frac{e^{iz}}{(1+e^{iz})^2 }\right|_{z=0} 
 =
 i\left.\left(\frac{d}{dz}\right)^k h\right|_{z=0}
 =
 i i^k \left.\left(\frac{d}{dz}\right)^k g\right|_{z=0}.
\end{equation}
We show that the right most expression in \eqref{2_feb_2020_14_24} is $0$. We have 
$$
g(z)= \frac{1}{1+e^z}= \frac{1}{2}-\frac{1}{2} \frac{e^{z/2}-e^{-z/2}}{e^{z/2}+e^{-z/2}}=\frac{1}{2}-\frac{1}{2}\tanh\frac{z}{2}.
$$
For $k\geq 2$, 
$$
\left(\frac{d}{dz}\right)^k g=\left(\frac{d}{dz}\right)^k \left(\frac{1}{2}-\frac{1}{2}\tanh\frac{z}{2}\right)
=0-\frac{1}{2}\left(\frac{d}{dz}\right)^k  \tanh\frac{z}{2}.
$$
But as $z\mapsto \tanh \frac{z}{2}$ is an odd function, it follows that for even $k$, 
$$
z\mapsto \left(\frac{d}{dz}\right)^k  \tanh\frac{z}{2}
$$ 
is an odd function too, and 
in particular, it vanishes at $z=0$. Hence $\left.\left(\frac{d}{dz}\right)^k g\right|_{z=0}=0$, and using \eqref{2_feb_2020_14_24}, also 
$$
\left.\left(\frac{d}{dt}\right)^{k-1}\frac{ e^{it}}{(1+e^{it})^2}\right|_{t=0}=0.
$$
This completes the proof. 
\end{proof}

\section{Appendix A: Casimir effect}
 
\noindent In quantum field theory, upon quantising a classical field,
one ends up with infinitely many harmonic oscillators, one at each
spacetime point.  If we take this picture seriously, then we run into
the problem of having to add up all of their ground state energies,
and taking that as the ground state energy of the quantum field. The
Casimir effect, predicted in 1948 \cite{Cas}, allows the experimental
demonstration \cite{Ede} of the existence of this ground state energy obtained by
summing the ground state energies of all the oscillators.
\begin{figure}[h]
     \center 
     \psfrag{0}[c][c]{${\scriptstyle 0}$}
     \psfrag{d}[c][c]{${\scriptstyle d}$}
     \psfrag{x}[c][c]{${\scriptstyle x}$}
     \includegraphics[width=3 cm]{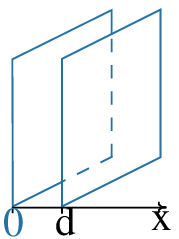}
\end{figure}
 
\noindent
Consider two parallel uncharged large plane conductors of area $A$
separated by a small distance $d$ in empty space, as shown.  This is
like a capacitor, but the plates do not have any charge on them. From
the classical point of view, there should not be any electromagnetic
force between them. Nevertheless, it can be demonstrated
experimentally that the two plates attract each other. This can be
explained as follows. In the absence of the plates, the quantum
electromagnetic field (a superposition of an infinite number of
oscillators) is in its ground state. But upon the introduction of the
two plates, we impose perfect conductor boundary conditions for the
electromagnetic field exactly at $x=0$ and $x=d$. So the ground state
energy for the oscillators will now be different (but still divergent)
from the original ground state energy.  The difference between the
ground state energies, calculated with and without the plates, will
involve subtracting one infinity from another, but seems to be
formally finite as we will see in the toy example below, and can be
used to predict the correct magnitude of the Casimir force, which has
been experimentally verified  \cite{Ede}.
 
As a simplified toy model, consider a massless scalar field $\varphi$
with only one space dimension, and suppose that the field is
constrained to vanish at $x=0$ and $x=d$. The allowed oscillator modes
$$
\varphi_n(x)=c_n \sin \frac{n\pi x}{d}, \quad n=0,1,2,3,\cdots
$$
have wave numbers given by
$$
k_n=\frac{n\pi}{d}.
$$

\noindent 
The corresponding travelling wave is then
$$
\varphi_n(x,t)= c_n \sin (k_n x-\omega_n t),\quad x,t\in \mR,
$$
with the speed of propagation 
$$
c=\frac{\omega_n}{k_n}.
$$
The associated ground state energy with a quantum harmonic oscillator
with angular frequency $\omega$ is
$$
E=\frac{\hbar \omega}{2}.
$$
Thus we now obtain that the total ground state energy in the
superposition of all possible modes is
\begin{eqnarray*}
E(d)&=&\sum_{n=1}^\infty \frac{\hbar \omega_n}{2}
\\
&=&\frac{\pi \hbar c}{2d}\sum_{n=1}^\infty n\\
&=&-\frac{\pi c \hbar}{24 d}.
\end{eqnarray*}
If we imagine two point particles at $x=0$ and at $x=d$ (analogous to
the two conducting plates), then this energy leads to an attractive
force of magnitude
$$
E'(d)= \frac{\pi c \hbar}{24 d^2}.
$$
 
\section{Appendix B: Zeta function regularisation} 
 
\noindent In order for easy comparison and contrast of our summation method versus 
the zeta regularisation method, both used in showing
\begin{equation}
\label{eq_17_01_2020_14_50}
\sum_{n=1}^\infty n =-\frac{1}{12},
\end{equation}
we outline briefly here the idea behind the zeta function regularisation method.  

\smallskip 

\noindent 
One first considers the
Riemann zeta function, given by
$$
\zeta(s) :=\sum_{n=1}^\infty n^{-s}.
$$
So we would like to set $s=-1$, in order to get the desired sum
$$
1+2+3+\cdots.
$$
The series for $\zeta(s)$ converges absolutely if
$\textrm{Re}(s)>1$, but it diverges whenever $\textrm{Re}(s)\leq 1$.
However, one can use an analytic continuation of the zeta function  
 in the
punctured plane $\mC\setminus \{1\}$ for determining the zeta function
at points where the series fails to converge. In \cite{Son}, it was
shown that a series derived using Euler's transformation provides the
analytic continuation of $\zeta$ for all complex numbers $s \neq 1$, and in particular at
negative integers, the series becomes a finite sum, whose value is given
by an explicit formula for Bernoulli numbers. In particular, this formula then
yields $\zeta(-1)=-1/12$.

\noindent Our summation method yields 
\begin{equation}
\label{17_jan_2020_1937}
1^k+2^k+3^k+\cdots= 
\frac{1}{i^{k-1}} \cdot\frac{1}{1- 2^{k+1}}\cdot \left.\left(\frac{d}{dt}\right)^{k-1}\frac{ e^{it}}{(1+e^{it})^2}\right|_{t=0}.
\end{equation}
We had seen in Proposition~\ref{prop_arne} of Section~\ref{section_7}, our formula \eqref{17_jan_2020_1937} gives values matching exactly the
corresponding values of $\zeta(-k)$. 

However, our route to arriving at $1^k+2^k+3^k+\cdots$ 
is quite different from the analytic continuation of the zeta
function, since we rely on distribution theory. The formula
\eqref{17_jan_2020_1937} we obtained is afforded in particular by
($\ast$), ($\star$) on page \pageref{17_jan_2020_1819}, and
\eqref{17_jan_2020_18_39} on page \pageref{17_jan_2020_18_39}, where
the numbers
$$
\left(\frac{d}{dt}\right)^{k} \frac{e^{it}}{(1+e^{it})^2}\Big|_{t=0} 
$$
appear. 
The Fourier series appears in quantum field theory computations, and hence
it is conceivable that our summation method is more natural in this
context. We refer the reader to \cite{KNW}, where the Casimir effect is 
analysed using the framework of quantum field theory, and in particular to \cite[\S III]{KNW}, 
where Fourier series plays a role.

 \medskip 
 
 \goodbreak 
 
\noindent {\bf Acknowledgement:} Useful discussions with Sara
Maad-Sasane, are gratefully acknowledged. I thank Arne Meurman for the proof of Proposition~\ref{prop_arne}. Thanks are also due to the
anonymous referee for the many insightful suggestions which greatly
improved the article. In particular, the query raised of
whether $T_0 \in \bsigma$, resulted in 
Proposition~\ref{prop_17_jan_2020_1443}.

\end{document}